\documentclass{article}
\usepackage[%
journal=JST,    %% Replace XXX by one of the above journal codes.
lang=british,   %% Change to `american' if you use American English.
]{ems-journal}

\numberwithin{equation}{section}

 \theoremstyle{plain}            % body italics
 \newtheorem{theorem}{Theorem}[section]
 \newtheorem{proposition}[theorem]{Proposition}
 \newtheorem{lemma}[theorem]{Lemma}

 \theoremstyle{definition}       % body roman

 \newtheorem{remark}[theorem]{Remark}

\newcommand{\ee}{\mathrm{e}}
\newcommand{\D}{\mathrm{d}}

\newcommand{\R}{\mathbb{R}}

\newcommand{\OO}{\mathcal{O}}

\newcommand{\dist}{\mathrm{dist}}
\newcommand{\pd}{\partial}
\newcommand{\sumstar}{\!\raisebox{1.5ex}{{\scriptsize $*$}}}

\newcommand{\col}[1]{{\color{black}{#1}}}

\begin{document}

%------
% Insert the title of your paper and (if necessary)
% a short title for the running head.
%------
\title{Bound States in Bent Soft Waveguides}
\titlemark{Bound States in Bent Soft Waveguides}

%------
% Insert full names of the authors.
% Add further authors as follows:
%  \emsauthor{2}{}{}
%  \emsauthor{3}{}{}
% etc.
%------
\emsauthor{1}{Pavel Exner}{P.~Exner}

%------
% Use \authormark if the list of authors is too
% long for the running head: \authormark{A.~Doe et al.}
%------

%------
% Add one \emsaffil and one \email for each author.
%------
\emsaffil{1}{Doppler Institute for Mathematical Physics and Applied Mathematics, Czech Technical University,  B\v rehov\'a 7, 11519 Prague, Czechia \\ and Department of Theoretical Physics, NPI, Academy of Sciences, 25068 \v{R}e\v{z} near Prague, Czechia \email{exner@ujf.cas.cz}}

%------
\emsauthor{1}{Semjon Vugalter}{S.~Vugalter}

%------
% Use \authormark if the list of authors is too
% long for the running head: \authormark{A.~Doe et al.}
%------

%------
% Add one \emsaffil and one \email for each author.
%------
\emsaffil{1}{Institute for Analysis, Karlsruhe Institute of Technology, Englerstrasse 2, 76131 Karlsruhe, Germany
\email{semjon.wugalter@kit.edu}}

%------
% Add MSC 2020 codes according to https://zbmath.org/classification/.
% A unique primary MSC code (in curly brackets) is mandatory,
% while secondary MSC codes (in square brackets) are optional.
%------
\classification[35J10, 35P15]{81Q37}

%------
% Add a list of keywords.
%------

\begin{dedication}
{To Brian Davies on the occasion of his 80th birthday}
\end{dedication}

%------
% Insert your abstract.
%------
\begin{abstract}
The aim of this paper is to show that a two-dimensional Schr\"odinger operator with the potential in the form of a `ditch' of a fixed profile can have a geometrically induced discrete spectrum; this happens if such a potential channel has a single or multiple bends being straight outside a compact. Moreover, under stronger geometric restrictions the claim remains true in the presence of a potential bias at one of the channel `banks'.
\end{abstract}

\maketitle

%%%%%%%%%%%%%%%%%%%%%%%%%%%%%%%%%%%%%%%%%%%%%%%%%%%%%%%%%%%%%%%%%%

\section{Introduction}
\label{s: intro}

Behavior of quantum particles confined to tubular regions attracted a lot of attention in the last decades with the motivation coming from two sources. On the physics side it was the possibility to use such models to describe a guided dynamics in various condensed matter systems. At the same time, this appeared to be a source of interesting mathematical problems, in particular, those concerning spectral effects coming from the geometry of the confinement; for an introduction to the topic and a bibliography we refer to the book \cite{EK15}.

There are different ways how to localize a particle in the configuration space. One possibility is a hard confinement where the Hamiltonian is typically the Dirichlet Laplacian associated with a tube in $\R^d$ (or more complicated regions such as layers, networks, etc.). From the point of view of application to objects like semiconductor wires such a model has a drawback; it does not take into account the tunneling between different parts the waveguide. This fact motivated investigation of the `leaky' confinement in which the Hamiltonian is instead Schr\"odinger operator with an attractive singular interaction supported by a curve (or a surface, metric graph, etc.); to have it well defined, the codimension of the interaction support must not exceed three.

If we stay for simplicity in the two dimensional situation, both models exhibit \emph{curvature-induced bound states}: whenever the strip, or the curve supporting the $\delta$ interaction, is non-straight but asymptotically straight, the corresponding Hamiltonian has a non-void discrete spectrum; this claim is valid universally modulo technical requirements on the regularity and asymptotic behavior.

Leaky guide model has another drawback in assuming that the interaction support has zero width. This motivated recently investigation of a more realistic situation when the potential in the Schr\"odinger operator is regular in the form of a channel of a fixed profile \cite{Ex20}. The term coined was \emph{soft waveguides}; the analogous problem was studied in three dimensions \cite{Ex22} as well as for soft layers \cite{EKP20, KK22}. One has to add that such operators were considered before \cite{EI01, WT14}, however, the focus was then on the limit in which the potential shrinks transversally to a manifold; in the physics literature the idea of determining the right `quantization' on a manifold through such a limit was examined a long time ago \cite{KJ71, To88}.

Not very surprisingly, soft waveguides were already shown to share properties with their hard and leaky counterparts, an example is the ground state optimization in a loop-shaped geometry \cite{EL21}. Some results have also been obtained concerning the problem we are interested in here, the existence of curvature-induced bound states, however, so far they lack the universal character indicated above. In \cite{Ex20} Birman-Schwinger principle was used to derive a sufficient condition under which the discrete spectrum is nonempty, expressed in terms of of positivity of a certain integral which, in general, is not easy to evaluate. An alternative is to apply the variational method; in this way the existence was established in the example of a particular geometry, often referred to as a `bookcover' \cite{KKK21}. We note in passing that it is paradoxically easier to establish the existence in conic-shaped soft layers, where the discrete spectrum is infinite \cite{EKP20, KK22}.

The trouble with the variational approach is that it is not easy, beyond the simple example mentioned, to find a suitable trial function. The aim of this paper is to extend the existence result using a variational method to a much wider, even if still not optimal class of soft waveguides. The main restrictions in our analysis are the limitation of the curved part into a bounded region, a compact support of the potential defining the channel profile, and the requirement of the profile symmetry. The latter restriction can be relaxed in some situations, in particular, if the profile potential is sign-changing and the transverse part of operator, the operator \eqref{profileSO} below, has zero-energy resonance.

We will also consider the situation when the system has a constant positive potential bias in one of the regions separated by the profile potential support. In this case we have a stronger geometry restriction: we have to assume that one of the two regions is \emph{convex}. If the bias potential is supported in it, we can again prove the existence of a discrete spectrum, even without the symmetry assumption. If the bias is supported in the opposite region, we have the existence again, however, except in the situation when the operator \eqref{profileSO} has a zero-energy resonance; this is in agreement with the result of \cite{EV16} where we treated a system which can be regarded as a singular version of the present system. Let us stress that the convexity makes it also possible to prove the existence in the absence of the bias and the symmetry restriction, provided that operator \eqref{profileSO} has a negative eigenvalue.

In the following section we will state the problem in proper terms and present the main results. The rest of the paper is devoted to the proofs. The next two sections deal with case without the bias; in Sec.~\ref{s: proof1} we prove part (a) of Theorem~\ref{th:main1} which concerns the situation when the operator \eqref{profileSO} has a zero-energy resonance, Sec.~\ref{s: proof2} provides the proof of part (a) of Theorem~\ref{th:main2} which addresses the case when the operator has a negative eigenvalue and the channel profile is symmetric. Finally, in Sec.~\ref{s: proof3} we prove parts (b) of the two theorems which establish the existence results in the situation when one of the two regions to which the potential channel, not necessarily symmetric, divides the plane is convex, even in the absence of the bias, except in the zero-energy resonance case.

%%%%%%%%%%%%%%%%%%%%%%%%%%%%%%%%%%%%%%%%%%%%%%%%%%%%%%%%%%%%%%%%%%

\section{Statement of the problem and main results}
\label{s: main}

Let us now state the problem described in the introduction. We begin with the assumptions which are split into two groups; the first one concerns the support of the potential, the other the channel profile. The former is a strip built around a curve $\Gamma$, understood as the graph of a function $\Gamma:\: \R\to\R^2$ such that $|\dot\Gamma(s)|=1$. Without repeating it at every occasion we always exclude the trivial situation when $\Gamma$ is a straight line; in addition to that we suppose:
 % ------------- %
 \begin{enumerate}[(s1)]
 \setlength{\itemsep}{1.5pt}
\item $\Gamma$ is \col{$C^1$, piecewise} $C^3$-smooth, non-straight but straight outside a compact; its curved part consists of a finite number of segments such that on each of them the monotonicity character of the signed curvature $\kappa(\cdot)$ of $\Gamma$ and its sign are preserved, \label{s1}
\item $|\Gamma(s_+)-\Gamma(s_-)|\to\infty$ as $s_\pm\to\pm\infty$, in other words, the two straight parts of $\Gamma$ are either not parallel, or if they are, they point in the opposite directions, \label{s2}
\item the strip neighborhood $\Omega^a := \{ x\in\R^2:\,\dist(x,\Gamma)<a \}$ of $\Gamma$ with a halfwidth $a>0$ does not intersect itself. \label{s3}
 \end{enumerate}
 % -------------- %
\begin{figure}[h!]
\centering
    \includegraphics[clip, trim=8cm 17cm 6cm 4cm, angle=0, width=0.5\textwidth]{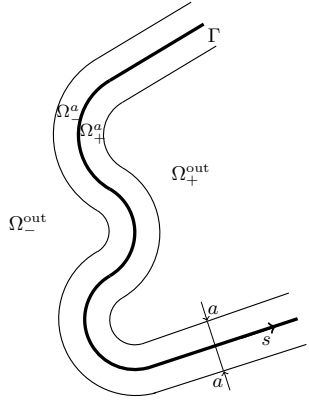}
\caption{Scheme of the waveguide}
\end{figure}
 % -------------- %

\noindent Assumption \ref{s3} has various equivalent expressions: one can say, for instance, that the function $\dist(x,\Gamma(\cdot))$ has for any fixed $x\in\Omega^a$ a unique minimum, or that the map
 % ------------- %
 \begin{equation} \label{strip}
x(s,t) \mapsto \big(\Gamma_1(s)-t\Gamma'_2(s), \Gamma_2(s)+t\Gamma'_1(s) \big)
 \end{equation}
 % ------------- %
from the straight strip $\Omega^a_0 := \R\times (-a,a)$ to $\R^2$ is a bijection, in fact, a diffeomorphism; $\vec n(s)=(-\Gamma'_2(s), \Gamma'_1(s))$ is, of course, the (inward) normal to the curve at the point $\Gamma(s)$. Under assumption \ref{s1}, the signed curvature $\kappa:\: \kappa(s)= (\Gamma'_2\Gamma''_1 - \Gamma'_1\Gamma''_2)(s)$ is \col{piecewise} smooth and compactly supported function; a necessary, but in general not sufficient condition for \ref{s3} to hold is $a\|\kappa\|_\infty<1$ which ensures the local injectivity of the map. The curve divides the plane into open regions which we denote as $\Omega_\pm$; for the sake of definiteness we assume that $\Omega_+$ is \emph{at the left side} when one looks in the direction of the increasing arc length variable $s$.

We also introduce $\Omega^a_\pm:= \Omega_\pm\cap\Omega^a$ so that we have $\Omega^a= \Omega^a_+\cup\Gamma\cup\Omega^a_-$; given our choice of the normal orintation,
the labels correspond to the sign of the transversal variable~$t$. Finally, we will use a natural symbol for the complement of the strip, namely $\Omega_\mathrm{out}:= \R^2 \setminus\Omega^a$, and its one-sided components will be denoted as $\Omega^\mathrm{out}_\pm:=\Omega_\pm\setminus\Omega^a$ -- cf.~Fig.~1.

The second group of assumptions concerns the potential. Its profile is determined by a function $v:\,\R\to\R$ of which we assume
 % ------------- %
 \begin{enumerate}[(p1)]
 \setcounter{enumi}{0}
 \setlength{\itemsep}{1.5pt}
 \item $v\in L^2(\R)$ and $\:\mathrm{supp}\,v \subset[-a,a]$\,; \label{p1}
  \end{enumerate}
 % ------------- %
in some situations, specifically in part (a) of Theorem~\ref{th:main2} below, we will require it additionally to be mirror-symmetric,
 % ------------- %
 \begin{enumerate}[(p1)]
 \setcounter{enumi}{1}
 \setlength{\itemsep}{1.5pt}
 \item $v(t)=v(-t)$ for $t\in[-a,a]$. \label{p2}
 \end{enumerate}
 % ------------- %
In addition to the potential defining the channel we are going to consider, in general, also a one-sided potential bias of the system. To this aim, we introduce the one-dimensional Schr\"odinger operator
 %----------------%
\begin{equation}\label{profileSO}
 h:= -\frac{\D^2}{\D t^2} + v(t) + V_0\chi_{[a,\infty)}(t), \quad V_0\ge 0.
\end{equation}
 %----------------%
The crucial role will be played by the spectral bottom of this operator, specifically we will be concerned with the following two possibilities:
 % ------------- %
 \begin{enumerate}[(p1)]
 \setcounter{enumi}{2}
 \setlength{\itemsep}{1.5pt}
 \item $\inf\sigma(h)$ is a negative (ground state) eigenvalue $\mu$ associated with a real-valued eigenfunction $\phi_0$ which we may without loss of generality normalize by the requirement $\phi_0(-a)=1$, \label{p3}
 \item operator $h$ has a zero-energy resonance \col{(virtual level)}, meaning that $h\ge 0$ and $-(1-\varepsilon)\frac{\D^2}{\D t^2} + v(t) + V_0\chi_{[a,\infty)}(t)$ has a negative eigenvalue for any $\varepsilon>0$. In this case, the equation $h\phi=0$ has a \col{solution $\phi_0 \in \dot{H}^1(\R)$ where $\dot{H}^1(\R)$ is the Hilbert space with the norm $\|\phi\|^2:= \int_{|x|\le 1} |\phi(x)|^2 \D x + \int_{\R} |\phi'(x)|^2 \D x\,$, cf.~\cite{BBV21}. Since $\mathrm{supp}\,v \subset [-a,a]$, this solution  is nonzero and constant for $|t|>a$ in the absence of the bias, for $V_0>0$ this is true for $t<-a$ only}. The solution will be again supposed to satisfy the normalization condition $\phi_0(-a)=1$. \label{p4}
 \end{enumerate}
 % ------------- %

The main object of our interest is the Schr\"odinger operator
 %----------------%
\begin{subequations}
\label{Hamiltonian}
\begin{equation}\label{operator}
 H_{\Gamma,V} = -\Delta +V(x)
\end{equation}
 %----------------%
on $L^2(\R^2)$ with the potential defined using the locally orthogonal coordinates $(s,t)$ appearing in \eqref{strip} as
 %----------------%
\begin{equation}\label{potential}
 V(x) = \left\{ \begin{array}{cl}  v(t) & \qquad\text{if}\;\: x\in\Omega^a \\[.5em] V_0  & \qquad\text{if}\;\: x\in\Omega_+\!\setminus\Omega^a \\[.5em] 0 & \qquad\text{otherwise} \end{array} \right.
\end{equation}
 %----------------%
We will often drop the subscript of $H_{\Gamma,V}$ if it is clear from the context.

 %----------------%
\begin{proposition} \label{prop:ess}
Under the assumptions \ref{s1}--\ref{s3} \col{and} \ref{p1}, the operator \eqref{Hamiltonian} is self-adjoint, $D(H_{\Gamma,V})=H^2(\R^2)$, \col{and its essential spectrum is the same as for $\Gamma$ being a straight line,} $\sigma_\mathrm{ess}(H_{\Gamma,V})= [\mu,\infty)$. \col{The threshold $\mu$ is the lowest eigenvalue of $h$ if \ref{p3} is valid, while if the dicrete spectrum of $h$ is empty, we have} $\mu=0$.
\end{proposition}
 %----------------%
\begin{proof}
The self-adjointness is easy to check; it is sufficient to ascertain, using assumption \ref{p1}, that the potential \eqref{potential} is infinitely small with respect to $-\Delta$. \col{We will do it in the form sense using the KLMN theorem \cite[Sec.~X.1]{RS} by checking that to any $a>0$ there is a $b>0$ such that $(\psi,V\psi) \le a\|\nabla\psi\|^2 +b\|\psi\|^2$ holds for all $\psi\in H^1(\R^2)$.} Suppose first that $V_0=0$. We decompose any \col{such} $\psi$ into a sum $\psi=\psi_+ + \psi_0 + \psi_-$ of \col{$H^1$} functions such that $\mathrm{supp}\,\psi_\pm\upharpoonright_{\Omega^a}$ \col{contain} the straight parts of the strip and $\mathrm{supp}\,\psi_0\upharpoonright_{\Omega^a}$ \col{is bounded and} contains the curved part. To the latter the theorem applies directly since  $V\upharpoonright_{\mathrm{supp}\,\psi_0}\in L^2$. In the straight parts we get first using \ref{p1} the one-dimensional version of the inequality in the transverse variable, then we lift it to two dimensions \col{using the inequality $\|\partial_t\psi\|^2 \le \|\nabla\psi\|^2$}. The constants $b_j,\, j=0,\pm,$ in the obtained inequalities are in general different; we put \col{$b:=3\max\{b_+,b_0,b_-\}$}. Using then the triangle and Schwarz inequalities which, in particular, give \col{$\|\psi\|^2 \le 3\sum_j \|\psi_j\|^2$}, we arrive at the desired conclusion. Finally, the self-adjointness is \col{certainly} not affected by adding the bounded potential $V_0\chi_{\Omega^\mathrm{out}_+}$.

The identification of the essential spectrum of $H_{\Gamma,V}$ with the interval $[\mu,\infty)$, where $\mu=\inf\sigma(h)$, was established in \cite[Proposition~3.1]{Ex20} under slightly different assumptions. The argument can be easily modified for our present purpose; the requirement on the smoothness of $\Gamma$ we made is stronger than there, and neither the substitution of a bounded negative $v$ by a possibly sign indefinite square integrable one, nor the addition of a potential bias alters the conclusion.
\end{proof}

Note also that the above Hamiltonian can be investigated using the associate quadratic form $Q_{\Gamma,V}$, mostly written without the indices specifying the curve and the potential, and defined by
 %----------------%
\begin{equation}\label{form}
Q[\psi] := \|\nabla\psi\|^2 + \int_{\Omega} V(x) |\psi(x)|^2\, \D x,\quad D(Q)=H^1(\R^2)\,;
\end{equation}
 %----------------%
\end{subequations}
 %----------------%
in the \col{absence of the potential bias the potential of the profile $v$ is integrated over $\Omega^a$, while for $V_0>0$ we have in addition $V_0 \int_{\Omega_+^\mathrm{out}} |\psi(x)|^2\, \D x$ on the right-hand side.}

Now we are in position to state our main results. The assumptions may appear in various combinations; we group them according the according to the spectral threshold $\mu$ starting with the situation when operator \eqref{profileSO} has a zero-energy resonance:

 %----------------%
\begin{theorem}[threshold resonance case] \label{th:main1}
Assume \ref{s1}--\ref{s3}, \ref{p1} and \ref{p4}; then the following claims are valid: \\[.2em]
(a) If the bias is absent, $V_0=0$, and
 %----------------%
\begin{equation} \label{alt_assumpt}
[\phi_0(a)^2-\phi_0(-a)^2] \int_\R \kappa(s)\, \D s \le 0
\end{equation}
 %----------------%
holds, then $H_{\Gamma,V}$ has at least one negative eigenvalue. \\[.2em]
(b) The same is true if $V_0>0$ and  $\Omega_+$ is \emph{convex}.
\end{theorem}
 %----------------%

 %----------------%
\begin{remark} \label{rem:wmirror}
Recall that $\kappa$ does not vanish identically. The condition \eqref{alt_assumpt} is naturally satisfied if $\phi_0(a)=\phi_0(-a)$, in particular, under the mirror-symmetry assumption \ref{p2}. Consider further the asymmetric situation, $\phi_0(a)\ne\phi_0(-a)$, and recall that the integral in \eqref{alt_assumpt} equals $\pi-2\theta$ where $2\theta$ is the angle between the asymptotes. Consequently, at least one bound state exists then in the zero-energy resonance case if the asymptotes of $\Gamma$ are parallel and pointing in the opposite directions, $\theta=\frac12\pi$, or if they are not parallel and the resonance solution $\phi_0$ is larger at the `outer' side of the strip~$\Omega^a$.
\end{remark}
 %----------------%

\noindent If $h$ has negative eigenvalues so that $\mu<0$, the situation is more complicated and we have to make stronger restrictions on the profile or the shape of the waveguide:
 %----------------%
\begin{theorem}[eigenvalue case] \label{th:main2}
Assume \ref{s1}--\ref{s3} together with \ref{p1} and \ref{p3}. Then $\sigma_\mathrm{disc}(H_{\Gamma,V})$ is nonempty under any of the following conditions: \\[.2em]
(a) $V_0=0$ and assumption \ref{p2} is satisfied. \\[.2em]
(b) $V_0\ge 0$ and \emph{one of the regions} $\Omega_\pm$ is \emph{convex}.
\end{theorem}
 %----------------%

%%%%%%%%%%%%%%%%%%%%%%%%%%%%%%%%%%%%%%%%%%%%%%%%%%%%%%%%%%%%%%%%%%

\section{Proof of Theorems~\ref{th:main1} \col{and \ref{th:main2}} -- the first part}
\label{s: proof1}

With the later purpose in mind we will formulate the argument first in the general situation which involves both the bound-state and zero-energy-resonance cases as well as the possible potential bias. In view of Proposition~\ref{prop:ess}, it is sufficient to construct a trial function $\psi\in H^1(\R^2)$ such that $Q[\psi]<\mu\|\psi\|^2$. Let us first fix the geometry. If the two straight parts of $\Gamma$ are not parallel -- cf.~Fig.~1 -- their line extensions intersect at a point which we choose as the origin $O$, and use polar coordinates with this center, in which the two halflines correspond to the angles $\pm\theta_0$ for the appropriate $\theta_0\in(0,\frac12\pi)$. Furthermore, we fix the point $s=0$ in such a way that for large $|s|$ the points with the coordinates $\pm s$ have the same Euclidean distance from~$O$ \col{-- cf.~Fig.~2.}
 % -------------- %
\begin{figure}[h!]
\centering
    \includegraphics[clip, trim=8cm 17cm 6cm 4cm, angle=0, width=0.5\textwidth]{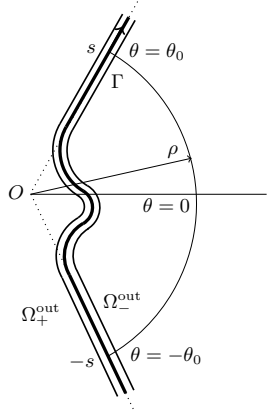}
\caption{Parallel and polar coordinates used in the proof}
\end{figure}
 % -------------- %

If the asymptotes are parallel (and pointing in the opposite directions according to \ref{s2}), we choose the origin as the point with equal distance from the endpoints of the two halflines. The point with $s=0$ on the curve is likewise chosen so that those with the coordinates $\pm s$ have the same Euclidean distance from the origin; in both cases one can check easily that such a choice is unique.

%%%%%%%%%%%%%%%%%%%%%%%%%%%%%%%%%%%%%%%%%%%%
\subsection{Trial function inside the strip}
\label{ss:trial inside}

For fixed values $s_0$, such that the points with coordinates $\pm s_0$ lay outside the curved part of $\Gamma$, and $s^*>s_0$, to be chosen later, we define
 %----------------%
\begin{equation}\label{mollif}
 \chi_\mathrm{in}(s) := \left\{ \begin{array}{cl}  1 & \qquad\text{if}\;\: |s|<s_0 \\[.5em] \ln\frac{s^*}{|s|} \big(\ln\frac{s^*}{s_0} \big)^{-1}  & \qquad\text{if}\;\: s_0\le |s|\le s^* \\[.5em] 0 &\qquad\text{if}\;\: |s|> s^* \end{array} \right.
\end{equation}
 %----------------%
Recalling that $\phi_0$ is the ground-state eigenfunction or the zero-energy solution normalized by $\phi_0(-a)=1$, we put
 %----------------%
\begin{equation}\label{trial_in}
 \psi(s,t) = \phi_0(t)\chi_\mathrm{in}(s) + \nu g(s,t),\quad |t|\le a,
\end{equation}
 %----------------%
where the parameter $\nu$ and the function $g$, compactly supported within $(-s_0,s_0)\times(-a,a)$, will be chosen later. We denote by $Q_\mathrm{int}[\psi]$ the contribution to the shifted quadratic form, $Q[\psi]-\mu\|\psi\|^2$, coming from the strip $\Omega^a$, which can be using the parallel coordinates expressed as
 %----------------%
\begin{align*}
 Q_\mathrm{int}[\psi] =& \int_{|t|\le a} \Big\{ \Big(\frac{\pd\psi}{\pd s}\Big)^2 (1-\kappa(s)t)^{-1} + \Big(\frac{\pd\psi}{\pd t}\Big)^2 (1-\kappa(s)t) \nonumber \\[.3em]
 & + (v(t)-\mu) |\psi|^2 (1-\kappa(s)t) \Big\}\, \D s\D t. %\label{int_form}
\end{align*}
 %----------------%
The first term on the right-hand side can be estimated as
 %----------------%
$$ %\begin{equation}\label{est1}
 \int_{|t|\le a} \Big(\frac{\pd\psi}{\pd s}\Big)^2 (1-\kappa(s)t)^{-1}\, \D s\D t \le 2\tau_0^{-1} \|\phi_0\!\upharpoonright_{[-a,a]}\|^2 \|\chi'_\mathrm{in}\|^2 +C\nu^2,
$$ %\end{equation}
 %----------------%
where the norm refers to $L^2(\R)$, $\,\tau_0:= 1-a\|\kappa\|_\infty$ is positive by \ref{s3}, and $C$ depends on the function $g$ only; we will use the same letter for generic constants in the following. Note that choosing the parameter $s^*$ in \eqref{mollif} large one can make \col{for any fixed $s_0$} the norm $\|\chi'_\mathrm{in}\|= \big(\ln\frac{s^*}{s_0} \big)^{-1} \big(\frac{1}{s_0} - \frac{1}{s^*}\big)^{1/2}$ small. As for the other two terms, we have
 %----------------%
\begin{align}
 & \int_{|t|\le a} \Big\{\Big(\frac{\pd\psi}{\pd t}\Big)^2 (1-\kappa(s)t)  + (v(t)-\mu) |\psi|^2 (1-\kappa(s)t) \Big\}\, \D s\D t \nonumber \\[.3em]
 & = \int_{|t|\le a} \big\{ (\phi_0'(t))^2 + (v(t)-\mu)|\phi_0(t)|^2\big\} \chi_\mathrm{in}^2(s) (1-\kappa(s)t)\, \D s\D t \nonumber \\[.3em]
 & \quad + 2\nu \int_{|t|\le a} \Big\{ \phi_0'\frac{\pd g}{\pd t} + (v(t)-\mu)\phi_0 g \Big\} \chi_\mathrm{in}(s) (1-\kappa(s)t)\, \D s\D t \nonumber \\[.3em]
 & \quad + \nu^2 \int_{|t|\le a} \Big\{ \Big(\frac{\pd g}{\pd t}\Big)^2 + (v(t)-\mu)|g|^2 \Big\} (1-\kappa(s)t)\, \D s\D t, \label{two_terms}
\end{align}
 %----------------%
where the last term on right-hand side can be again estimated by $C\nu^2$ with a $C$ depending on the function $g$ only. Furthermore, integrating the middle term by parts with respect to $t$, we get
 %----------------%
\begin{align}
 & 2\nu \int_{|t|\le a} \big[ -\phi_0'' + (v(t)-\mu)\phi_0\big] \chi_\mathrm{in}(s) g(s,t)(1-\kappa(s)t)\, \D s\D t \nonumber \\[.3em]
 & \col{+}\,2\nu \int_{|t|\le a} \phi_0'(t) \chi_\mathrm{in}(s) g(s,t)\kappa(s)\, \D s\D t, \label{middle}
\end{align}
 %----------------%
where the square bracket in the first integral is zero by assumption.

Notice next that $\phi_0'$ cannot vanish identically in the interval $[-a,a]$. Indeed, it is continuous in $\R$ and we have $v(t)=0$ for $|t|>a$, hence should the derivative $\phi_0'$ be zero in $[-a,a]$, the function must have been a constant one, however, that is impossible for an eigenfunction or a zero-energy resonance solution. This observation allows us to choose the function $g$ in such a way that the last integral is \col{negative. To this aim, it suffices to have it supported in a region where both $\phi_0'$ and $\kappa$ do not change sign -- recall that $\chi_\mathrm{in}(s)=1$ holds on the support of $\kappa$ -- and to pick the sign of $g(s,t)$ accordingly.} With such a choice the expression \eqref{middle} will \col{equal to $-\delta\nu$ where $-\delta<0$ is the value of the last integral. For} small $\nu$ this linear term will dominate over those estimated by $C\nu^2$, \col{and consequently, there is a $\nu_0>0$ such that for all $\nu\in(0,\nu_0)$ and any $s_0,\,s^*$ indicated in \eqref{mollif} the sum of the last two terms at the right-hand side of \eqref{two_terms} will be smaller than $-\frac12\delta\nu$. Note that for a fixed $v$ determining $\phi_0$ and $\mu$, and fixed $g(s,t)$, the value $\nu_0$ depends continuously on $\kappa$. Hence if we have two curves $\Gamma$ and $\tilde\Gamma$ the curvatures of which do not differ much, the corresponding values of $\nu_0$ are close to each other too; we will use this fact later in Sec.~4.3.}

It remains to deal with the first term on the right-hand side of \eqref{two_terms}. To simplify the notation, we introduce the following symbols,
 %----------------%
\begin{equation} \label{shorthands}
\phi_+=\phi_0(a),\quad \xi_+=-\sqrt{|\mu|+V_0},\quad \xi_-=\sqrt{|\mu|}.
\end{equation}
 %----------------%
which allows us to write $\phi'_0(a)=\xi_+\phi_+$ and $\phi'(-a)=\xi_-$; recall that $\phi_0(-a)=1$ holds by assumption. The expression in question then can be rewritten using integration by parts as follows:
 %----------------%
\begin{align}
 & \int_{|t|\le a} \big\{ (\phi_0'(t))^2 + (v(t)-\mu)|\phi_0(t)|^2\big\} \chi_\mathrm{in}^2(s) (1-\kappa(s)t)\, \D s\D t \nonumber \\[.3em]
 & = \int_\R \big[ \xi_+\phi_+^2(1-\kappa(s)a) - \xi_-(1+\kappa(s)a)\big] \chi_\mathrm{in}^2(s)\, \D s \nonumber \\[.3em]
 & \quad + \int_{|t|\le a} \big\{ -\phi_0''(t) + (v(t)-\mu)\phi_0(t)\big\} \phi_0(t) (1-\kappa(s)t) \chi_\mathrm{in}^2(s) \, \D s\D t \nonumber \\[.3em]
 & \quad + \int_{|t|\le a} \col{\phi_0'(t)\phi_0(t)}\, \kappa(s)\, \chi_\mathrm{in}^2(s) \, \D s\D t \nonumber \\[.3em]
 & = \big[ \xi_+\phi_+^2 - \xi_-\big] \|\chi_\mathrm{in}\|^2 - \big[ \xi_+\phi_+^2 + \xi_-\big]\,a \int_\R \kappa(s) \chi_\mathrm{in}^2(s) \, \D s \nonumber \\[.3em]
 & \quad + \frac12 (\phi_+^2-1) \int_\R \kappa(s) \chi_\mathrm{in}^2(s) \, \D s, \label{firstterm}
\end{align}
 %----------------%
where the norm in the last expression refers to $L^2(\R)$ and we have used the identity $\phi_0'\phi_0 = \frac12(\phi_0^2)'$. Since $\kappa$ has a compact support and $\chi_\mathrm{in}^2(s)=1$ holds on it by \eqref{mollif}, we can replace the integrals in the last part of \eqref{firstterm} by $\int_\R \kappa(s)\,\D s$. Summarizing the estimate, we have obtained for all sufficiently small $\nu$ the inequality
 %----------------%
\begin{align}
 Q_\mathrm{int}[\psi] \le & \:-\frac12\delta\nu + \big[ \xi_+\phi_+^2 - \xi_-\big] \|\chi_\mathrm{in}\|^2 - \big[\xi_+\phi_+^2 + \xi_-\big]\,a \int_\R \kappa(s)\, \D s \nonumber \\[.3em]
 & \quad + \frac12 (\phi_+^2-1) \int_\R \kappa(s)\, \D s + \tau_0^{-1} \|\phi_0\!\upharpoonright_{[-a,a]}\|^2 \|\chi'_\mathrm{in}\|^2. \label{inner}
\end{align}
 %----------------%
Choosing \col{for a fixed $\nu\in(0,\nu_0)$ and $\tau_0=1-a\|\kappa\|_\infty$ the ratio} at the right-hand side of \eqref{mollif} \col{as $\frac{s^*}{s_0}\gg1$}, one can achieve that the last term in \eqref{inner} will be smaller than $\frac14\delta\nu$.

\col{This general estimate simplifies in various situations indicated above. In particular, the assumptions \ref{p2} and $V_0=0$ used in part (a) of Theorem~\ref{th:main2} imply $\phi_+=\phi_-=1$ and $\xi_-=-\xi_+$, so that \eqref{inner} becomes}
 %----------------%
\begin{equation}\label{inner_sym}
 Q_\mathrm{int}[\psi] \le -\frac14\delta\nu - 2|\mu|^{1/2} \|\chi_\mathrm{in}\|^2
\end{equation}
 %----------------%
\col{for all $\nu\in(0,\nu_0)$, the ratio $\frac{s^*}{s_0}$ large enough, and $s_0$ such that the points referring to the curved part of $\Gamma$ are inside $[-s_0,s_0]$. On the other hand, assumption \ref{p4} used in part (a) of Theorem~\ref{th:main1} together with the absence of the bias, $V_0=0$, means that $\xi_\pm=0$, so the second and the third term on the right-hand side of \eqref{inner} vanish, and since the fourth one is now supposed to be non-positive, the estimate reduces simply to $Q_\mathrm{int}[\psi] \le -\frac14\delta\nu$.

%%%%%%%%%%%%%%%%%%%%%%%%%%%%%%%%%%%%%%%%%%%%%%%%%%%%%%%%%%%%%%%%%%%%%%%%%%
\col{\subsection{Completing the proof of part (a) Theorem~\ref{th:main1}}}
\label{ss:compl_2.2a}

\col{In view of the last observation, to conclude the proof of part (a) of Theorem~\ref{th:main1} we have thus} to choose the outer part of trial function in such a way that its contribution to the quadratic form can be made smaller than any fixed positive number. Under assumption \ref{p4}, $\phi_0$ is constant outside the potential support, $\phi_0(t)=\phi_\pm$ for $\pm t\ge a$; we recall that $\phi_-=1$ by assumption. We choose the trial function outside $\Omega^a$ as the $\phi_\pm$ multiplier of the} mollifier $\chi_\mathrm{out}$ \col{of which we require} the following properties:
 % ------------- %
 \begin{enumerate}[(i)]
 \setcounter{enumi}{0}
 \setlength{\itemsep}{1.5pt}
 \item in $\R^2\setminus\Omega^a$ the function \col{$\chi_\mathrm{out}$} depends on $\rho=\mathrm{dist}(x,O)$ only,
 \item we have continuity at the boundary: at the points $x(s,\pm a)$ the relation $\chi_\mathrm{out}(x)=\chi_\mathrm{in}(s)$ holds.
 \end{enumerate}
 % ------------- %
Let us consider the situation where the extensions of the asymptotes of $\Gamma$ cross; the case of parallel asymptotes pointing in the opposite directions can dealt with analogously. We again choose $s_0$ in such a way that the points $\Gamma(\pm s_0)$ belong to the straight parts of the curve, then $\mathrm{dist}(\Gamma(s),O) = \rho_s:= (|s|-s_0)+d_0$, where $d_0=\mathrm{dist}(\Gamma(s_0),O)$ (recall that $\Gamma(-s_0)=\Gamma(s_0)$ holds by assumption).

Given that the distance of the points $x(s,\pm a)$ from the origin is $\sqrt{\rho_s^2+a^2}$, in accordance with the requirements (i), (ii) we put
 %----------------%
$$ %\begin{equation}\label{mollif_out}
 \chi_\mathrm{out}(\rho) := \left\{ \begin{array}{cl}  \chi_\mathrm{in}(\sqrt{\rho^2-a^2}-d_0+s_0) & \qquad\text{if}\;\: \sqrt{\rho^2-a^2}\ge d_0 \\[.5em] 1  & \qquad\text{if}\;\: \rho \le \sqrt{d_0^2+a^2} \end{array} \right.
$$ %\end{equation}
 %----------------%
This, in particular means, that $\chi_\mathrm{out}$ vanishes if its argument exceeds~$s^*$, in other words, for $\rho>\sqrt{(s^*-s_0+d_0)^2+a^2}$. Since $\mu=0$ holds by assumption and the potential is zero away from $\Omega^a$, the quantity to be estimated is the kinetic energy contribution to the form \eqref{form} from the outer part of the trial function,
 %----------------%
\begin{align}
 & \int_{\Omega\setminus\Omega^a} |\nabla\psi_\mathrm{out}(x)|^2 \D x \label{kinetout} \\[.3em]
 & \;\; \col{\,\le\,} 2\pi\,\col{\max\{\phi_+^2,1\}} \int_{\sqrt{d_0^2+a^2}}^{\sqrt{(s^*-s_0+d_0)^2+a^2}} \Big| \frac{\D}{\D\rho} \chi_\mathrm{in}(\sqrt{\rho^2-a^2}-d_0+s_0)\Big|^2 \!\rho\D\rho. \nonumber
\end{align}
 %----------------%

Relation \eqref{inner_sym} tells us that one can choose parameters $\delta$ and $\nu$ for which the inner contribution to the form is negative (using a sufficiently large $s^*$), hence to prove the claim it is enough to show that the integral on the right-hand side of \eqref{kinetout} vanishes if $s_0,\,d_0\to\infty$ with the difference $s_0-d_0$ bounded and $\frac{s^*}{s_0}\to\infty$. The values of the integrated function on the support of $\nabla\psi_\mathrm{out}$ can be expressed using \eqref{mollif} to be
 %----------------%
\begin{align*}
 & \Big| \frac{\D}{\D\rho} \chi_\mathrm{in}(\sqrt{\rho^2-a^2}-d_0+s)\Big|^2 = \Big| \Big(\ln\frac{s^*}{s_0} \Big)^{-1} \frac{1}{\sqrt{\rho^2-a^2}-d_0+s_0}\, \frac{\pd s}{\pd\rho} \Big|^2 \nonumber \\[.3em]
 & \hspace{10em} = \Big(\ln\frac{s^*}{s_0} \Big)^{-2}\, \big(\sqrt{\rho^2-a^2}-d_0+s_0\big)^{-2},
\end{align*}
 %----------------%
because $\frac{\pd s}{\pd\rho}=1$ in the considered region. Substituting from here to \eqref{kinetout} we get
 %----------------%
\begin{align*}
 & \int_{\Omega\setminus\Omega^a} |\nabla\psi_\mathrm{out}(x)|^2 \D x \nonumber \\[.3em] & \quad \col{\,\le\,} 2\pi\,\col{\max\{\phi_+^2,1\}} \Big(\ln\frac{s^*}{s_0} \Big)^{-2} \int_{\sqrt{d_0^2+a^2}}^{\sqrt{(s^*-s_0+d_0)^2+a^2}} \hspace{-.5em} \frac{\rho\D\rho}{(\sqrt{\rho^2-a^2}-d_0+s_0)^2} %\label{est2}
\end{align*}
 %----------------%
Since $s_0-d_0$ is bounded and $a$ is fixed, we can choose a \col{$c\in(0,1)$} in such a way that
 %----------------%
$$
\sqrt{\rho^2-a^2}-d_0+s_0 \ge \col{c\rho}
$$
 %----------------%
holds for all \col{$\rho\ge\sqrt{d_0^2+a^2}$}, in which case we have
 %----------------%
\begin{equation} \label{est3}
\int_{\Omega\setminus\Omega^a} |\nabla\psi_\mathrm{out}(x)|^2 \D x \le \col{2\pi c^{-2}}\,\col{\max\{\phi_+^2,1\}} \Big(\ln\frac{s^*}{s_0} \Big)^{-2} \ln\rho\,\Big|_{\sqrt{d_0^2+a^2}}^{\sqrt{(s^*-s_0+d_0)^2+a^2}}.
\end{equation}
 %----------------%
Using again the fact that $s^*,\,d_0\to\infty$ while $a$ is fixed and $s_0-d_0$ bounded, we see that the parameters can be chosen so that
 %----------------%
\begin{equation} \label{est4}
\ln\frac{\sqrt{(s^*-s_0+d_0)^2+a^2}}{\sqrt{d_0^2+a^2}} \le \col{\ln\frac{c^*s^*}{c_0 s_0} = \ln\frac{s^*}{s_0} +\ln c^* - \ln c_0}
\end{equation}
 %----------------%
\col{for $c^*>1$ and $c_0\in(0,1)$. Substituting} from \eqref{est4} into \eqref{est3}, we get the needed result; this concludes the proof of part (a) of Theorem~\ref{th:main1}.

%%%%%%%%%%%%%%%%%%%%%%%%%%%%%%%%%%%%%%%%%%%%%%%%%%%%%%%%%%%%%%%%%%

\col{\section{Completing the proof of part (a) of Theorem~\ref{th:main2}}}
\label{s: proof2}

Let us pass to the situation where there is again \emph{no bias}, $V_0=0$, the channel profile is \emph{symmetric}, and the transverse operator \eqref{profileSO} is \emph{subcritical}, $\mu<0$.

%%%%%%%%%%%%%%%%%%%%%%%%%%%%%%%%%%%%%%%%%%%%%%%%%%%
\col{\subsection{Trial function outside $\Omega_a$}
\label{ss:main1outside}

This is the most difficult part of the argument. For} the interior we can use the result of Sec.~\ref{ss:trial inside} noting that in view of the assumption \ref{p2} we have $\phi_+=1$ and $\xi_+=-\xi_-$ which means that the inequality \eqref{inner_sym} is still valid. To begin with, we \col{introduce} in $\Omega_\mathrm{out}$ function $\phi$ \col{defined as}
 %----------------%
\begin{equation} \label{outer}
\phi(x) := \exp\{ -\xi(\mathrm{dist}(x,\Gamma)-a)\}, \quad x\in\R^2\setminus\Omega^a,
\end{equation}
 %----------------%
where $\xi:=\xi_-=-\xi_+=|\mu|^{1/2}$. The sought trial function will be then of the form $\psi_\mathrm{out}=\phi\chi_\mathrm{out}$ with the mollifier $\chi_\mathrm{out}$ to be specified below. As before, we will focus on the situation where the asymptotes of $\Gamma$ are not parallel, the case with $\theta_0=\frac{\pi}{2}$ can be treated in a similar way.

Since $\theta_0>0$ by assumption, we can choose conical neighborhoods of the asymptotes which do not intersect, that is, to pick $\Delta\theta_0$ sufficiently small so that $[-\theta_0-\Delta\theta_0,-\theta_0+\Delta\theta_0] \cap [\theta_0-\Delta\theta_0,\theta_0+\Delta\theta_0] = \emptyset$. Furthermore, we pick an $r_0>0$ large enough to ensure that the curved part of $\Gamma$ is contained in the disk of one half that radius, $B_{\frac12 r_0}(O)$, centered at the coordinate origin $O$. \col{In addition, we assume that $s_0$ in \eqref{mollif} is sufficiently large so that the parts of $\Gamma$ with $|s|>s_0$ are outside the disk $B_{r_0}(O)$ of the doubled radius.} At the points of the corresponding conical sectors, $x=(\rho,\theta)\in \R^2\setminus B_{r_0}(O)$ with $\theta\in[\theta_0-\Delta\theta_0,\theta_0+\Delta\theta_0]$ or $\theta\in[-\theta_0-\Delta\theta_0,-\theta_0+\Delta\theta_0]$ we can use the $(s,t)$ coordinates \col{-- cf. Fig.~3 --} and define the mollifier $\chi_\mathrm{out}$ depending on the longitudinal variable only,
 %----------------%
$$
\chi_\mathrm{out}(s,t) = \chi_\mathrm{in}(s),
$$
 %----------------%
where the right-hand side is given by \eqref{mollif}. Furthermore, at the points $x\in B_{r_0}(O)\setminus\Omega^a$ we put $\chi_\mathrm{out}(x) = 1$, and finally, in the remaining part of the plane we choose $\chi_\mathrm{out}$ independent of $\theta$, in other words, as a function of the distance $\rho$ from the origin $O$ only, and such that $\chi_\mathrm{out}$ is continuous in $\Omega_\mathrm{out}$. It is clear that the radial decay of such an external mollifier is determined by the behavior of the function \eqref{mollif}.
 % -------------- %
\begin{figure}[h!]
\centering
    \includegraphics[clip, trim=7cm 13cm 6cm 4cm, angle=0, width=0.5\textwidth]{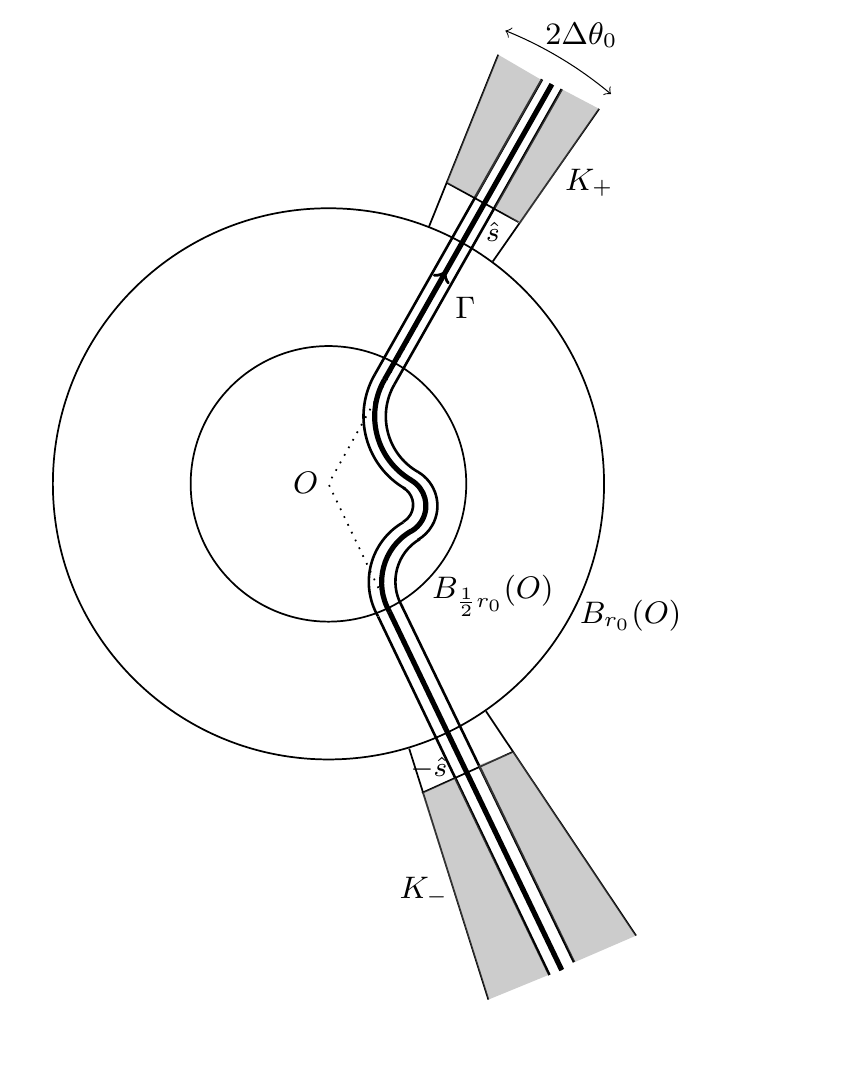}
\caption{The regions used in the proof of Theorem~\ref{th:main2}}
\end{figure}
 % -------------- %

Since the potential is supported in $\Omega^a$, the contribution to the quadratic form \eqref{form} in the exterior region comes from the kinetic term only. The trial function factorizes into a product and our first goal is to show that the cross-term containing the integral of $2\nabla\phi \cdot \nabla\chi_\mathrm{out}$ is small for large $r_0$, in particular, that one can make it smaller than $\frac{1}{16}\delta\nu$ with respect to the quantities appearing in \eqref{inner_sym}.

 %----------------%
\begin{lemma} \label{l:smallcross}
We have
 %----------------%
\begin{align}
& \int_{\Omega_\mathrm{out}} |\nabla\psi_\mathrm{out}(x)|^2 \D x \le \int_{\Omega_\mathrm{out}} |\nabla\phi(x)|^2 \chi^2_\mathrm{out}(x)\, \D x \nonumber \\[.3em] &  \qquad + \int_{\Omega_\mathrm{out}} |\phi(x)|^2 |\nabla\chi_\mathrm{out}(x)|^2 \D x + \OO(r_0^{-1}) \quad\text{as}\;\; r_0\to\infty.
\label{crossterm}
\end{align}
 %----------------%
\end{lemma}
 %----------------%
\begin{proof}
Since $\psi_\mathrm{out}=\phi\chi_\mathrm{out}$ for $x\in\Omega_\mathrm{out}$, we have to estimate the integral $\int_{\Omega_\mathrm{out}} |\nabla\phi(x) \cdot \nabla\chi_\mathrm{out}(x)|\,\D x$ to deal with the cross-term. To this aim we first note that \mbox{$\chi_\mathrm{out}=1$} holds inside $B_{r_0}(O)$ so we have to consider only the complement of the disk. In the conical sectors of $\R^2\setminus B_{r_0}(O)$ with $[\pm\theta_0-\Delta\theta_0,\pm\theta_0+\Delta\theta_0]$ the point nearest to $(x,\theta)$ lies on the straight part of $\Gamma$, as the distance to it is at most $\rho\Delta\theta<\frac12\rho$ while that to the curved part is at least $\rho-\frac12 r_0>\frac12\rho$. This implies that $\nabla\phi$ is perpendicular to $\nabla\chi_\mathrm{out}$, and the corresponding contribution to the integral vanishes too. Finally, in view of our definition of $\chi_\mathrm{out}$ in combination with \eqref{mollif} we see that $\nabla\chi_\mathrm{out}$ is bounded outside $B_{r_0}(O)$ and the two sectors, and furthermore, we have $|\nabla\phi(x)| \le \xi\phi(x)\le C\, \ee^{-\rho/2}$ which yields
 %----------------%
$$ %\begin{equation} \label{crossest}
\int_{\Omega_\mathrm{out}} |\nabla\phi(x) \cdot \nabla\chi_\mathrm{out}(x)|\,\D x \le C' \int_{\Omega_\mathrm{out}} |\nabla\phi(x)|\,\D x = \OO(r_0^{-1})
$$ %\end{equation}
 %----------------%
as $r_0\to\infty$ which we set out to prove.
\end{proof}

Let us turn to the second term on the right-hand side of \eqref{crossterm}.
 %----------------%
\begin{lemma} \label{l:smallsecond}
We have
 %----------------%
\begin{equation} \label{smallsecond}
\int_{\Omega_\mathrm{out}} |\phi(x)|^2 |\nabla\chi_\mathrm{out}(x)|^2 \D x = \OO(r_0^{-1}) \quad\text{as}\;\; r_0\to\infty.
\end{equation}
 %----------------%
\end{lemma}
\begin{proof}
The integral over the disk \col{$B_{r_0}(O)$} is again zero and using an argument analogous to that of the previous proof, one can check that the integral over the region outside the conical sectors is $\OO(r_0^{-1})$ as $r_0\to\infty$. Inside the sectors we have
 %----------------%
$$
|\phi(x)|^2 |\nabla\chi_\mathrm{out}(x)|^2 = |\phi_0(t)|^2 |\chi'_\mathrm{in}(s)|^2 \le |\chi'_\mathrm{in}(s)|^2
$$
 %----------------%
with $\chi_\mathrm{in}$ given by \eqref{mollif}; recall that outside $\Omega^a$ the function $\phi_0$ decays exponentially with the distance from $\Gamma$ and $\phi_0(\pm a)=1$ holds by assumption. Hence the integral in \eqref{smallsecond} can be estimated by the squared norm of $\chi'_\mathrm{in}$, and since to a given $r_0$ the value of \col{$s_0$ was chosen so that the parts of $\Gamma$ with $|s|>s_0$ lay outside $B_{r_0}(O)$, one can pick an} $s^*=s^*(r_0)$ in such a way that $\ln\frac{s^*}{s_0} > Cr_0$ for some $C>0$, \col{and} the claim follows.
\end{proof}

\col{Combining these two lemmata we see that for any fixed $\nu\in(0,\nu_0)$ one can choose $r_0,\,s_0$, and $\frac{s^*}{s_0}$ sufficiently large to satisfy the inequality}
 %----------------%
\begin{equation} \label{est5}
\int_{\Omega_\mathrm{out}} |\nabla\psi_\mathrm{out}(x)|^2 \D x \le \int_{\Omega_\mathrm{out}} |\nabla\phi(x)|^2 \chi^2_\mathrm{out}(x)\, \D x + \frac{|\mu|}{8} \delta\nu.
\end{equation}
 %----------------%
Next we note that in part (a) of Theorem~\ref{th:main2} the bias is absent, $V_0=0$, which means that the function \eqref{outer} satisfies
 %----------------%
$$
|\nabla\phi|^2-\mu|\phi|^2 = 2|\nabla\phi|^2
$$
 %----------------%
almost everywhere in $\Omega_\mathrm{out}$. This means that we can estimate the whole exterior contribution to the form $Q[\psi]-\mu\|\psi\|^2$ by doubling the kinetic term and neglecting the one containing the eigenvalue $\mu$. \col{Combining then the estimates \eqref{inner_sym} and \eqref{est5} we further infer} that in order to prove the theorem it is sufficient to check that
 %----------------%
$$ %\begin{equation} \label{est6}
2\int_{\Omega_\mathrm{out}} |\nabla\phi(x)|^2 \chi^2_\mathrm{out}(x)\, \D x \le 2|\mu|^{1/2} \|\chi_\mathrm{in}\|^2 + \frac{|\mu|}{8} \delta\nu
$$ %\end{equation}
 %----------------%
\col{holds for some $\nu\in(0,\nu_0)$ and sufficiently large $r_0,\, s_0$ and $\frac{s^*}{s_0}$ depending on the values of $\delta$ and $\nu$. Moreover, this} is in view of $|\nabla\phi|^2=-\mu|\phi|^2$ further equivalent to
 %----------------%
\begin{equation} \label{est7}
\int_{\Omega_\mathrm{out}} |\phi(x)\chi_\mathrm{out}(x)|^2\, \D x \le |\mu|^{-1/2} \|\chi_\mathrm{in}\|^2 + \frac1{16} \delta\nu.
\end{equation}
 %----------------%

The rest of the proof consists of verification of the inequality \eqref{est7}. To begin with, we estimate the contribution to its left-hand side from the parts of the plane adjacent to the straight parts of the waveguide; we choose them as conical sectors similar to those used in the proof of Lemma~\ref{l:smallcross}. We recall that for $x=(\rho,\theta)$ with $\rho\ge r_0$ and $\theta\in[\pm\theta_0-\Delta\theta_0,\pm\theta_0+\Delta\theta_0]$ we can use the $(s,t)$ coordinates simultaneously with the polar ones. We choose an $\hat{s}\ge r_0$ so that the parts of $\Gamma$ with $|s|\ge\hat{s}$ lay outside $B_{r_0}(O)$, and at the same time we choose $s_0$ of \eqref{mollif} is such a way that $s_0>\hat{s}$. Then we define
 %----------------%
\begin{equation} \label{Ksector}
K_\pm:= \big\{\, x:\: |s|\ge\hat{s},\, |t|\ge a,\, \theta\in[\pm\theta_0-\Delta\theta_0,\pm\theta_0+\Delta\theta_0]\,\big\}
\end{equation}
 %----------------%
Within these sets, the closest points of $\Gamma$ are those on the straight parts of the curve with the same coordinate $s$. Then it is easy to see that
 %----------------%
\begin{equation} \label{est8}
\int_{\Omega_\mathrm{out}\cap\{K_+\cup K_-\}} |\phi(x)\chi_\mathrm{out}(x)|^2\, \D x \le |\mu|^{-1/2} \|\chi_\mathrm{in}\|^2_{L^2((-\infty,-\hat{s}]\cup [\hat{s},\infty))}
\end{equation}
 %----------------%

It remains to integrate the function $|\phi\chi_\mathrm{out}|^2$ over $\Omega_\mathrm{out}\setminus\{K_+\cup K_-\}$. Obviously, the integral will increase if we replace $\chi_\mathrm{out}$ by one, hence
to complete the proof, it is in view of \eqref{inner_sym} and \eqref{est8} enough to check \col{the following inequality for the fixed $\delta$ given by our choice of the function $g$, some $\nu\in(0,\nu_0)$, $r_0>0$ and $\Delta\theta_0>0$,}
 %----------------%
\begin{equation} \label{est9}
\int_{\Omega_\mathrm{out}\setminus\{K_+\cup K_-\}} |\phi(x)|^2\, \D x \le 2\hat{s}\,|\mu|^{-1/2} + \frac1{16} \delta\nu\,;
\end{equation}
 %----------------%
we have used here the fact that $\|\chi_\mathrm{in}\|^2_{L^2((-\hat{s},\hat{s}))} = 2\hat{s}$. \col{Let us note that it is easy to see that if \eqref{est9} holds for some $r_0$, the same is true for any $r'_0>r_0$.}

%%%%%%%%%%%%%%%%%%%%%%%%%%%%%%%%%%%%%%%%%%%%%%%%%%%%%%%%%%%%
\col{\subsection{Curves with a piecewise constant curvature}
\label{ss:piecewise}

We divide the verification of inequality \eqref{est9} into two parts, considering first a particular class of the generating curves assuming additionally that
 % ------------- %
 \begin{enumerate}[(s1)]
 \setlength{\itemsep}{1.5pt}
 \setcounter{enumi}{3}
\item \col{$\Gamma$ is a $C^1$ curve consisting of two halflines and a} \emph{finite array of circular arcs;} \col{we allow some of these arcs to be straight segments of a finite length.} \label{s4}
 \end{enumerate}
 % ------------- %
Consequently, the signed curvature $\kappa(\cdot)$ of such a curve is a step function, \col{being zero on the straight parts of $\Gamma$}.} To estimate the indicated integral \col{under} the additional assumption \ref{s4}; the two part of $\Gamma$ corresponding to $|s|>\hat{s}$ will be considered as arcs of zero curvature, cf.~Remark~\ref{rem:segment} below.
First of all, we note that the function $d_x:\R\to\R_+$, defined by $d_x(s):=\mathrm{dist}(x,\Gamma(s))$, is $C^1$ smooth for any $x\in\R^2$, and under the the assumption \ref{s4} it is piecewise monotonous because on each arc it can have at most one extremum. At the same time, $d_x(s)\to\infty$ holds as $s\to\pm\infty$, hence the function has a global minimum, positive as long as $x$ does not lie on the curve, and in view of its continuity it may also have a finite number of local extrema. \col{Among those, the extrema with largest and smallest value of $s$ are clearly both minima. As a result, the numbers of local extrema with $s<s_x^0$ and $s>s_x^0$ are both even; they come in pairs of an adjacent minimum and maximum. Let $s_x^1$ and $s_x^2$ be the coordinates of such a pair, respectively, then we have obviously $d_x(s_x^1) < d_x(s_x^2)$, and consequently
 %----------------%
\begin{equation} \label{est10a}
\exp\{-2\xi(d_x(s_x^1)-a)\} - \exp\{-2\xi(d_x(s_x^2)-a)\} \ge 0.
\end{equation}
 %----------------%
Let us denote by $M_x^\uparrow$ the subset of coordinates of the \emph{local} maxima and by $M_x^\downarrow$ the subset of coordinates of \emph{all} the minima. Summing the inequalities \eqref{est10a} over the pairs of \emph{local} extrema, we get
 %----------------%
$$ %\begin{equation} \label{est10b}
\sum_{\scriptsize\begin{array}{c}s_x^i\in M_x^\downarrow\\ s_x^i\ne s_x^0 \end{array}} \exp\{-2\xi(d_x(s_x^i)-a)\} - \sum_{s_x^i\in M_x^\uparrow} \exp\{-2\xi(d_x(s_x^i)-a)\} \ge 0.
$$ %\end{equation}
 %----------------%
Adding finally $\exp\{-2\xi(d_x(s_x^0)-a)\}$ to both sides of this inequality, we arrive at}
 %----------------%
\begin{align}
\exp\{-2\xi(d_x(s_x^0)-a)\} \le & -\!\!\sum_{s_x^i\in M_x^\uparrow} \exp\{-2\xi(d_x(s_x^i)-a)\} \nonumber \\ & +\!\! \sum_{s_x^i\in M_x^\downarrow} \exp\{-2\xi(d_x(s_x^i)-a)\}, \label{est10}
\end{align}
 %----------------%
for all $x\in\Omega_\mathrm{out}$, \col{where the last sum now runs over all the minima}. To estimate the integral in \eqref{est9}, we have to integrate the right-hand side of \eqref{est10} over $\Omega_\mathrm{out}\setminus\{K_+\cup K_-\}$. To this aim, let us first collect several simple geometric statements easy to check:
 % -------------- %
\begin{figure}[h!]
\centering
    \includegraphics[clip, trim=7cm 20cm 7cm 4cm, angle=0, width=0.8\textwidth]{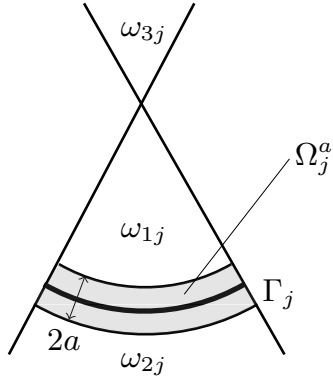}
\caption{The regions appearing in Proposition~\ref{p:geometry}}
\end{figure}
 % -------------- %
\begin{proposition} \label{p:geometry}
Let $\Gamma_j$ be one the arcs of $\Gamma$ and denote by $\omega_{1j}, \omega_{2j}, \omega_{3j}$ and $\Omega^a_j$ the open regions shown in Fig.~4. Then the following holds true:
\begin{enumerate}[(i)]
\item If $x\in\omega_{1j}\cup\omega_{2j}$, then $d_x(\cdot)$ has a minimum in the interior of~$\Gamma_j$.
\item If $x\in\omega_{3j}$, then $d_x(\cdot)$ has a maximum in the interior of~$\Gamma_j$.
\item $x\not\in\bar{\omega}_{1j}\cup\bar{\omega}_{2j}\cup\bar{\omega}_{3j}\cup \bar{\Omega}^a_j$, then $d_x(\cdot)$ has no extremum on~$\Gamma_j$.
\item $d_x(\cdot)$ cannot have more than one critical point in the interior of~$\Gamma_j$.
\item If $x\in\omega_{kj}$ for any of $k=1,2,3$, then the one-sided derivative $d_x'(s)\ne 0$ at the endpoints of $\Gamma_j$.
\end{enumerate}
\end{proposition}
 %----------------%
\begin{remark} \label{rem:segment}
With an abuse of terminology we include into \ref{s4} also situations when a $\Gamma_j$ is a straight segment, that is, $\kappa(s)=0$ holds on $\Gamma_j$. In that case the wedge-shaped regions $\omega_{1j}$ and $\omega_{2j}$ become semi-infinite strips and $\omega_{3j}$ does not exist. This concerns, in particular, the two straight parts of $\Gamma$ corresponding to $|s|>\hat{s}$.
\end{remark}
 %----------------%

Within the regions we introduced the minimal and maximal distances are easily expressed; we have
 %----------------%
\begin{align}
d_x(s_x^i) = \mathrm{dist}(x,\Gamma_j) & \quad\text{if}\;\; s_x^i\in \Gamma_j\cap M_x^\downarrow, \nonumber \\[-.75em] \label{est11} \\[-.75em]
d_x(s_x^i) = |\kappa_j|^{-1}+\mathrm{dist}(x,O_j) & \quad\text{if}\;\; s_x^i\in \Gamma_j\cap M_x^\uparrow, \nonumber
\end{align}
 %----------------%
where $O_j$ is the center of the corresponding circular arc.

Let $\iota_j^{1,2}$ and $\iota_j^3$ be the characteristic functions of the sets $\omega_{1j}\cup\omega_{2j}$ and $\omega_{3j}$, respectively. In view of of Proposition~\ref{p:geometry} and relations \eqref{est11}, we can replace the first term at the right-hand side of \eqref{est10}, everywhere except the zero measure set referring to the boundaries of the regions $\omega_{kj},\,k=1,2,3$, with
 %----------------%
\begin{equation} \label{est-sec3}
- \sum_j \exp\{-2\xi(|\kappa_j|^{-1}+\mathrm{dist}(x,O_j)-a)\}\iota_j^3(x)
\end{equation}
 %----------------%
and the second term similarly by
 %----------------%
\begin{equation} \label{est-sec12}
\sum_j \exp\{-2\xi(\mathrm{dist}(x,\Gamma_j)-a)\} \iota_j^{1,2}(x)
\end{equation}
 %----------------%
Integrating now \eqref{est-sec3} and \eqref{est-sec12} over $\Omega_\mathrm{out}\setminus\{K_+\cup K_-\}$ and exchanging the order of integration over $x$ with summation over $j$, we can using \eqref{est10} estimate \\ $\int_{\Omega_\mathrm{out}\setminus\{K_+\cup K_-\}} \exp\{-2\xi(d_x(s_x^0)-a)\}\,\D x$ from above by
 %----------------%
\begin{align}
& \sum_j \int_{(\omega_{1j}\cup\omega_{2j})\cap\{\Omega_\mathrm{out}\setminus\{K_+\cup K_-\}\}} \exp\{-2\xi(\mathrm{dist}(x,\Gamma_j)-a)\} \,\D x \label{est12} \\
& - \sum_j \int_{\omega_{3j}\cap\{\Omega_\mathrm{out}\setminus\{K_+\cup K_-\}\}} \exp\{-2\xi(|\kappa_j|^{-1}+\mathrm{dist}(x,O_j)-a)\} \,\D x, \nonumber
\end{align}
 %----------------%
where the sums run over all the indices of $\Gamma_j$ including those of the straight segments of the curve with $|s|>\hat{s}$. Note that this estimate includes in general a double counting since the same $x$ may belong to different $\omega_{kj}$; this does not matter as long as we consider the contributions referring of a given $\Gamma_j$ together.

Our next goal is to show that the expression \eqref{est12} cannot decrease if we replace the integration domains by $(\omega_{1j}\cup\omega_{2j}) \setminus\{K_+\cup K_-\}$ and $\omega_{3j}\setminus\{K_+\cup K_-\}$, respectively. To this aim, consider a fixed arc $\Gamma_{j_0}$ and the respective segment $\Omega^a_{j_0}$ of the strip $\Omega^a$ as indicated in Fig.~4. For a point $x\in\Omega^a_{j_0}$ the function $d_x(\cdot)$ has the global minimum on $\Gamma_{j_0}$ with a coordinate $s_x^0$ and all the local extrema, if they exist, come in pairs situated outside $\Gamma_{j_0}$. This yields the estimate
 %----------------%
$$ %\begin{equation}
0 \le -\!\!\sum_{s_x^i\in M_x^\uparrow} \exp\{-2\xi(d_x(s_x^i)-a)\} +\hspace{-1.2em} \sum_{\tiny \begin{array}{cc}s_x^i\in M_x^\downarrow \\ s_x^i\ne s_x^0 \end{array}} \!\! \exp\{-2\xi(d_x(s_x^i)-a)\} %\label{est-int1}
$$ %\end{equation}
 %----------------%
Using again Proposition~\ref{p:geometry}, we get for any $x\in\Omega^a_{j_0}$ the inequality
 %----------------%
\begin{align}
0 \le & - \sum_{j\ne j_0} \exp\{-2\xi(|\kappa_j|^{-1}+\mathrm{dist}(x,O_j)-a)\}\iota_j^3(x) \nonumber \\
+ & \sum_{j\ne j_0} \exp\{-2\xi(\mathrm{dist}(x,\Gamma_j)-a)\}) \iota_j^{1,2}(x)\,; \label{est-int2}
\end{align}
 %----------------%
since $x\in\Omega^a_{j_0}$, we are able to replace the indicator functions in this expression by their restriction $\iota_j^3 \upharpoonright \Omega^a_{j_0}$ and $\iota_j^{1,2} \upharpoonright \Omega^a_{j_0}$, respectively. Noting further that $\omega_{kj_0} \,\cap\, \Omega^a_{j_0} =\emptyset$ holds for $k=1,2,3$, we see that \eqref{est-int2} remains valid if the summation is taken over all the $j$'s. Integrating then the right-hand side over $x\in\Omega^a$ we arrive at the inequality
 %----------------%
\begin{align*}
0 \le & - \sum_j \int_{\omega_{3j}\cap \Omega^a_{j_0}} \exp\{-2\xi(|\kappa_j|^{-1}+\mathrm{dist}(x,O_j)-a)\} \,\D x \\
& + \sum_j \int_{(\omega_{1j}\cup\omega_{2j})\cap \Omega^a_{j_0}} \exp\{-2\xi(\mathrm{dist}(x,\Gamma_j)-a)\} \,\D x,
\end{align*}
 %----------------%
and summing this result over $j_0$ we get
 %----------------%
\begin{align}
0 \le & - \sum_j \int_{\omega_{3j}\cap \Omega^a} \exp\{-2\xi(|\kappa_j|^{-1}+\mathrm{dist}(x,O_j)-a)\} \,\D x \nonumber \\
& + \sum_j \int_{(\omega_{1j}\cup\omega_{2j})\cap \Omega^a} \exp\{-2\xi(\mathrm{dist}(x,\Gamma_j)-a)\} \,\D x. \label{est-int3}
\end{align}
 %----------------%
Combining now \eqref{est12} and \eqref{est-int3}, we obtain
 %----------------%
\begin{align}
& \int_{\Omega_\mathrm{out} \setminus\{K_+\cup K_-\}} |\phi(x)|^2\, \D x \label{est12a} \\ & \le \sum_j \int_{(\omega_{1j}\cup\omega_{2j})\setminus\{K_+\cup K_-\}} \exp\{-2\xi(\mathrm{dist}(x,\Gamma_j)-a)\} \,\D x \nonumber \\
& \; - \sum_j \int_{\omega_{3j}\setminus\{K_+\cup K_-\}} \exp\{-2\xi(|\kappa_j|^{-1}+\mathrm{dist}(x,O_j)-a)\} \,\D x. \nonumber
\end{align}
 %----------------%
The summation in \eqref{est12a} runs over all the curve segments including the straight ones. Let us first estimate the contribution of these infinite `arcs' to the positive part of \eqref{est12a} having in mind that in accordance with Remark~\ref{rem:segment} the segments with $\kappa=0$ do not contribute to the negative one. We denote by $\Gamma_+$ the segment with $s>\hat{s}$ and by $\omega_{1+}, \omega_{2+}$ the corresponding semi-infinite strips, then we have
 %----------------%
\begin{align}
& \int_{(\omega_{1+}\cup\omega_{2+}) \setminus\{K_+\cup K_-\}} \exp\{-2\xi(\mathrm{dist}(x,\Gamma_+)-a)\}\, \D x \label{est12b}  \\[.3em] & \qquad \le 2\int_{\omega_{1+} \setminus K_+} \exp\{-2\xi(\mathrm{dist}(x,\Gamma_+)-a)\}\, \D x \nonumber \\[.3em] &  \qquad = 2\int_{\rho(\hat{s})\cos\Delta\theta}^\infty \int_{s\sin\Delta\theta_0}^\infty \exp\{-2\xi(t-a)\}\, \D t \D s \nonumber \\[.3em] & \qquad = \frac{\ee^{2\xi a}}{4\xi^2 \sin\Delta\theta_0}\, \ee^{-\xi\sin2\Delta_0\theta \cdot \rho(\hat{s})},\nonumber
\end{align}
 %----------------%
where $\rho(\hat{s})$ is the distance of the point $\Gamma(\hat{s})$ to the origin. In view of our choice of $\hat{s}$ we have $\rho(\hat{s})\ge r_0$ and the integral at the right-hand side of \eqref{est12b} can be made arbitrarily small by choosing $r_0$ large enough. An analogous argument applies to the segment of $\Gamma$ with $s<-\hat{s}$.

Denote now by $\sum_j^*$ the sum over all the $\Gamma_j$ except of $\Gamma_\pm$. The conclusion just made allows us to replace the sum $\sum_j$ in \eqref{est12b} by $\sum_j^*$ with an error which can be made arbitrarily small by choosing an appropriately large $r_0$. Furthermore, we note that the positive part of \eqref{est12} cannot decrease if we enlarge the integration domain in all the integrals there replacing $(\omega_{1j}\cup\omega_{2j})\setminus\{K_+\cup K_-\}$ by $\omega_{1j}\cup\omega_{2j}$.

Our next goal is to argue that we can do the same in the negative part of \eqref{est12} replacing $\omega_{3j}\setminus\{K_+\cup K_-\}$ by $\omega_{3j}$. In such a case, of course, the corresponding change of the integrals goes in the wrong way; our aim is to show that it again produces an error which can be made small if $r_0$ is large. Indeed, regions $\omega_{3j}$ exist only for the curved segments of $\Gamma$ and those are by assumption inside $B_{\frac12 r_0}(O)$, while the regions $K_\pm$ are outside $B_{r_0}(O)$. Consequently, the contributions from the extended integration domains are
 %----------------%
\begin{align}
& \int_{\omega_{3j}\cap\{K_+\cup K_-\}} \exp\{-2\xi(|\kappa_j|^{-1}+\mathrm{dist}(x,O_j)-a)\} \,\D x \label{estK3} \\
& \quad \le \ee^{2\xi a}\, |\Gamma_j| \int_{\rho\ge\sqrt{3}r_0/2}^\infty \ee^{-\sqrt{3}\,\xi\rho}\,\rho\,\D\rho = |\Gamma_j|\, \mathcal{O}(\ee^{-3\xi r_0/2}) \nonumber
\end{align}
 %----------------%
uniformly in $j$, and since the length of the curved part is finite, the error coming from the extension of the integration domain is $\mathcal{O}(\ee^{-3\xi r_0/2})$. Combining \eqref{estK3} with \eqref{est12a} we get
 %----------------%
\begin{align}
& \int_{\Omega_\mathrm{out} \setminus\{K_+\cup K_-\}} |\phi(x)|^2\, \D x \label{est12c} \\ & \le \sum_j\sumstar \int_{\omega_{1j}\cup\omega_{2j}} \exp\{-2\xi(\mathrm{dist}(x,\Gamma_j)-a)\} \,\D x \nonumber \\ & \; - \sum_j\sumstar \int_{\omega_{3j}} \exp\{-2\xi(|\kappa_j|^{-1}+\mathrm{dist}(x,O_j)-a)\} \,\D x + \mathcal{O}(\ee^{-3\xi r_0/2}). \nonumber
\end{align}
 %----------------%
It is not difficult to evaluate the integrals appearing at the right-hand side of \eqref{est12c}: we have

 %----------------%
\begin{align}
\int_{\omega_{2j}} & \exp\{-2\xi(\mathrm{dist}(x,O_j)-a)\} \,\D x = \Big(\frac{1}{2\xi} + \frac{a|\kappa_j|}{2\xi} + \frac{|\kappa_j|}{4\xi^2}\Big)|\Gamma_j| \nonumber\\
& = \frac{|\Gamma_j|}{2\xi} + \frac{a}{2\xi} \int_{\Gamma_j} |\kappa(s)|\,\D s + \frac{1}{4\xi^2} \int_{\Gamma_j} |\kappa(s)|\,\D s \label{int2}
\end{align}
 %----------------%
and
 %----------------%
\begin{align*}
\int_{\omega_{1j}} & \exp\{-2\xi(\mathrm{dist}(x,O_j)-a)\} \,\D x = \frac{|\Gamma_j|}{2\xi} - \frac{a}{2\xi} \int_{\Gamma_j} |\kappa(s)|\,\D s \nonumber\\
& - \frac{1}{4\xi^2} \int_{\Gamma_j} |\kappa(s)|\,\D s + \frac{1}{4\xi^2} \int_{\Gamma_j} \ee^{-2\xi(|\kappa(s)|^{-1}-a)}|\kappa(s)|\,\D s.  %\label{int1}
\end{align*}
 %----------------%
for the positive part of the estimate, while in the negative one we use
 %----------------%
$$ %\begin{equation} \label{int3}
\int_{\omega_{3j}} \exp\{-2\xi(\kappa_j^{-1}+\mathrm{dist}(x,O_j))\} \,\D x = \frac{1}{4\xi^2} \int_{\Gamma_j} \ee^{-2\xi(|\kappa(s)|^{-1}-a)}|\kappa(s)|\,\D s.
$$ %\end{equation}
 %----------------%
Summing finally the contributions from given $\Gamma_j$ we get $|\Gamma_j|\xi^{-1}$, hence the expression \eqref{est12} is smaller that $2|\mu|^{-1/2}\hat{s} + o(r_0)$ which according to inequality \eqref{est9} proves part (a) of Theorem~\ref{th:main2} under the additional assumption~\ref{s4}.

%%%%%%%%%%%%%%%%%%%%%%%%%%%%%%%%%%%%%%%%%%%%%%%%%%%%%%%%%%%%%%%%%%%%%%%%%%%%%%
\col{\subsection{Extension to the case when condition \ref{s4} does not hold}}

\col{Before turning to this task, let us recall how the parameters of the trial function constructed above depend on each other. Assume that the potential $v$, and thus $\mu$ and $\phi_0$, and also $g(t,s)$ are fixed, then the parameters $\delta$ and $\nu_0$ introduced in Sec.~\ref{ss:trial inside} depend on the curvature only. Note that this fact justifies \emph{a posteriori} the possibility to use a fixed function $g(s,t)$, since having a family of curves the curvatures of which do not differ much, one can certainly find a common interval of the variable $s$ on which these curvatures do not change sign. As for $\nu$, it can be chosen as any number in $(0,\nu_0)$ noting, however, that if $\nu\to 0$, the parameters $r_0,\,s_0,\,s^*$ and $\hat{s}$ which we pick after fixing $\nu$ must tend to infinity. The number $\Delta\theta_0$ in Sec.~\ref{ss:main1outside} depends on the geometry of $\Gamma$ only, more specifically on the angle between the asymptotes of $\Gamma$.

Proceeding with the choice of the parameters, we pick $r_0>0$ depending on $\nu$ and $\Delta\theta_0$; it must be sufficiently large so that the curved part of $\Gamma$ is contained in $B_{\frac12r_0}(O)$ and, at the same time, the error in \eqref{est9} coming from Lemmata~\ref{l:smallcross} and \ref{l:smallsecond} does not exceed $\frac{1}{16}\delta\nu$. Next, to the chosen $r_0$ we pick $\hat{s}$ and $s_0>\hat{s}$ so that the parts of $\Gamma$ with $|s|>\hat{s}$ are outside $B_{r_0}(O)$, and finally, we pick an $s^*\gg s_0$ depending on $s_0$ and $\nu$.}

\col{After this preliminary, let us turn to the proof} that \eqref{est9} remains valid without the assumption \ref{s4}. We will use the same trial function as before, in particular, its outer part will be again of the form $\psi_\mathrm{out}=\phi\chi_\mathrm{out}$ with $\phi$ given \eqref{outer}; the idea is to approximate the curve $\Gamma$ satisfying \ref{s1} by curves with a piecewise constant curvature, \col{of} the same length and \col{with} the same halfline asymptotes. Specifically, we are going to employ the following \col{approximation} result:
 %----------------%
\begin{theorem}[Sabitov-Slovesnov \cite{SS10}] \label{th:ss10}
Let $\Gamma$ be a $C^3$-smooth curve \col{of a finite length} consisting of a finite number of segments such that on each of them the monotonicity character of the signed curvature $\kappa(\cdot)$ of $\Gamma$ and its sign are preserved. Then $\Gamma$ can be approximated by a $C^1$-smooth function $\hat\Gamma$ of the same length, the curvature of which is piecewise constant having jumps at the points $s_1<s_2< \cdots < s_N$, in the sense that the estimates
 %----------------%
\begin{equation} \label{ss_approx}
\|\Gamma^{(m)} - \hat\Gamma^{(m)}\|_\infty \le C\,\max_{1\le k\le N-1}(s_{k+1}-s_k)^{3-m}, \quad m=0,1,2,
\end{equation}
 %----------------%
hold with some $C>0$ for the function $\Gamma$ and its two first derivatives. \col{The endpoints of $\Gamma$ and $\hat\Gamma$ coincide, and the same is true for the tangent vectors at these points, and moreover, on each subinterval $(s_k,s_{k+1})$ the curvature $\hat\kappa$ of $\hat\Gamma$ satisfies the inequality $\min_{s\in(s_k,s_{k+1})} \kappa(s) \le \hat\kappa(s) \le \max_{s\in(s_k,s_{k+1})} \kappa(s)$.}
\end{theorem}
 %----------------%
It is obvious that the hypotheses of Theorem~\ref{th:ss10} are satisfied under our assumption~\ref{s1} \col{on the $C^3$ smooth parts of $\Gamma$.}

\col{Let $\{\Gamma_n\}$ by a sequence of curves with the following properties:
 %----------------%
\begin{enumerate}[(i)]
\item $\Gamma_n$ coincides with $\Gamma$ at all parts of the curve where $\kappa(s)=0$, in particular, all the $\Gamma_n$'s have the same asymptotes as $\Gamma$,
\item on the curved parts of $\Gamma$ the $\Gamma_n$'s approximate $\Gamma$ in the sense of Theorem~\ref{th:ss10} with $\max_k|s_{k+1}-s_k|<\frac{1}{n}$.
\end{enumerate}
 %----------------%
Recall that the curved part of $\Gamma$ has finite length which allows us to use Theorem~\ref{th:ss10}, and note also that we are abusing notation here using for the approximating curve the symbol which coincides with the one which in the previous section meant the arcs of a curve satisfying assumption \ref{s4}.

Let $\hat\delta_n$ and $\hat\nu_{0,n}$ be the quantities corresponding to $\hat\Gamma_n$ in the same way as $\delta$ and $\nu_0$, respectively, correspond to $\Gamma$. It is easy to see from \eqref{ss_approx} and the point (ii) above $|\hat\kappa_n(s)-\kappa(s)|<Cn^{-1}$ holds for some $C>0$ independent of $s$, which implies  $\hat\delta_n\to\delta$ and $\hat\nu_{0,n}\to\nu_0$ as $n\to\infty$. By $\Omega_\mathrm{out}^{(n)}$ we denote the region exterior to the strip of halfwidth $a$ built around $\hat\Gamma_n$. The expression on the left-hand side of \eqref{est9} can be in this case rewritten as
 %----------------%
\begin{equation} \label{approx1}
\int_{\Omega_\mathrm{out}\setminus\{K_+\cup K_-\}} |\phi(x)|^2\, \D x = \int_{\Omega_\mathrm{out}^{(n)}\setminus\{K_+\cup K_-\}} |\phi(x)|^2\, \D x + \int_{\Omega_\mathrm{out}\setminus\Omega_\mathrm{out}^{(n)}} |\phi(x)|^2\, \D x
\end{equation}
 %----------------%
Since $\Gamma$ and $\hat\Gamma_n$ differ on a bounded interval of $s$ only, and since $|\Gamma(s)-\hat\Gamma_n(s)|_\infty\to 0$ holds as $n\to\infty$ by Theorem~\ref{th:ss10}, uniformly in $s$, in the second term on the right-hand side we integrate the function $|\phi(x)|^2\le 1$ over a region the measure of which tends to zero as $n\to\infty$. For the first integral on the right-hand side of \eqref{approx1} we have
 %----------------%
\begin{align}
\int_{\Omega_\mathrm{out}^{(n)}\setminus\{K_+\cup K_-\}} |\phi(x)|^2\, \D x \le &\: \ee^{2\xi\|\Gamma-\hat\Gamma_n\|_\infty}
\int_{\Omega_\mathrm{out}^{(n)}\setminus\{K_+\cup K_-\}} \ee^{-2\xi\{\mathrm{dist}(x,\hat\Gamma_n)-a\}}\, \D x \nonumber \\
\le & \big[1+\varepsilon_n^{(1)}\big]\Big[ 2\hat{s}\,|\mu|^{-1/2} + \frac1{16} \hat\delta_n\hat\nu_n \Big], \label{approx2}
\end{align}
 %----------------%
where $\nu_n\in(0,\nu_{0,n})$ is fixed and $\varepsilon_n^{(1)}\to 0$ as $n\to\infty$; we have used the fact that \eqref{est9} is valid for $\hat\Gamma_n$ as the latter satisfies condition \ref{s4} and the approximation preserves by Theorem~\ref{th:ss10} the length of the curve. Since the points $s=\pm \hat{s}$ are on the straight parts of $\Gamma$ (coinciding with those of $\hat\Gamma_n$), the distance between them along both $\Gamma$ and $\hat\Gamma_n$ is the same being equal to $2\hat{s}$.

Moreover, to derive \eqref{est9} for curves satisfying condition \ref{s4} we required $r_0$ and $\hat{s}$ which depend on $\nu$ to be large. For the approximating curves $r_0$ depends on $\hat\nu_n$ but we avoid the situation when $\hat\nu_n\to 0$ which would imply $r_0\to\infty$ as mentioned in the opening of this section. For the sake of definiteness we fix $\hat\nu_n = \frac12\hat\nu_{0,n}$ and consider $n$ sufficiently large to have $\hat\nu_{0,n} \ge \frac12\nu_0$. Then we can choose $\hat{s}$ in \eqref{approx2} independent of $n$, which yields
 %----------------%
\begin{equation} \label{approx3}
\int_{\Omega_\mathrm{out}^{(n)}\setminus\{K_+\cup K_-\}} |\phi(x)|^2\, \D x \le 2\hat{s}\,|\mu|^{-1/2} + \frac1{16} \hat\delta_n\,\frac14\nu_0 + \varepsilon_n^{(2)},
\end{equation}
 %----------------%
where $\varepsilon_n^{(2)}\to 0$ as $n\to\infty$. Combining this with the fact that $\hat\delta_n\to \delta$ for $n\to\infty$ as mentioned above, we infer from \eqref{approx3} that relation \eqref{est9} holds for $\nu=\frac18\nu_0$ and all $n$ large enough; this concludes the proof.}

%%%%%%%%%%%%%%%%%%%%%%%%%%%%%%%%%%%%%%%%%%%%%%%%%%%%%%%%%%%%%%%%%%

\section{Concluding proofs of Theorems~\ref{th:main1} and Theorem~\ref{th:main2}}
\label{s: proof3}

It remains to establish parts (b) of both the main results. Let us begin with Theorem~\ref{th:main2} and prove it in the situation when $\Omega_+$ is \emph{convex}; by assumption \ref{p3} we have $\mu<0$. Inside $\Omega^a$ we choose the trial function as in the previous proofs so that inequality \eqref{inner} is valid for any $V_0\ge 0$; recall that we derived it assuming the presence of a bias. Moreover, picking a suitable coordinate $s^*\gg s_0$ at the right-hand side of \eqref{mollif}, the last term in \eqref{inner} can be made as before smaller than $\frac14\delta\nu$. Outside the strip $\Omega^a$ we set
 %----------------%
\begin{equation} \label{outer2}
\phi(x) := \phi_\pm \exp\{ -|\xi_\pm|(\mathrm{dist}(x,\Gamma)-a)\} \quad \text{if}\;\; x\in\Omega_\pm^\mathrm{out},
\end{equation}
 %----------------%
recalling that $\phi_-=1$, which is a natural generalization of \eqref{outer}, and we employ the same mollifier $\chi_\mathrm{out}$ as before, cf.~Sec.~\ref{ss:piecewise}. Repeating the argument of this section, we arrive at the inequality \eqref{est5}, however, now with the function $\phi$ given by \eqref{outer2}. Let us split the outer contribution to the quadratic form into two parts referring, respectively, to $\Omega_\pm^\mathrm{out}$, for which we have
 %----------------%
\begin{subequations}
\label{outform}
\begin{align}
 & Q_\mathrm{out}^{(+)}[\psi_\mathrm{out}] = \int_{\Omega_+^\mathrm{out}} |\nabla\psi_\mathrm{out}(x)|^2\,\D x + \int_{\Omega_+^\mathrm{out}} (V_0-\mu) |\psi_\mathrm{out}(x)|^2\,\D x \nonumber \\
 & \qquad \le \int_{\Omega_+^\mathrm{out}} \big\{ |\nabla\phi(x)|^2 + (V_0-\mu)|\phi(x)|^2\big\} \chi_\mathrm{out}(x)^2 \,\D x + \frac{1}{16}\delta\nu, \label{outform+} \\
 & Q_\mathrm{out}^{(-)}[\psi_\mathrm{out}] \le \int_{\Omega_-^\mathrm{out}} \big\{ |\nabla\phi(x)|^2 -\mu|\phi(x)|^2\big\} \chi_\mathrm{out}(x)^2 \,\D x + \frac{1}{16}\delta\nu \label{outform-}
\end{align}
\end{subequations}
 %----------------%
in view Lemmata~\ref{l:smallcross} and \ref{l:smallsecond} provided that $r_0$ is chosen large enough. As in the proof of the first part of Theorem~\ref{th:main1} we choose an $\hat{s}\in[r_0,s_0)$ for which the parts of $\Gamma$ with $|s|\ge\hat{s}$ are outside $B_{r_0}(O)$ and use the regions $K_\pm$ defined by \eqref{Ksector}. By the definition \eqref{outer2} we have
 %----------------%
$$
|\nabla\phi|^2 = \xi_\pm^2 |\phi|^2 \quad\text{for}\;\; x\in\Omega_\pm^\mathrm{out}.
$$
 %----------------%
Within $\Omega_\mathrm{out} \cap\{K_+\cup K_-\}$ we may use the $(s,t)$ coordinates, and noting the $\phi$ is independent of $s$ there, and as a function of $t$ it coincides with the eigenfunction $\phi_0$ of $h$, cf. \eqref{profileSO}, we get
 %----------------%
\begin{subequations}
\label{Kest}
\begin{align}
\int_{\Omega_+^\mathrm{out} \cap\{K_+\cup K_-\}} & \big\{ |\nabla\phi(x)|^2 + (V_0-\mu) |\phi(x)|^2\big\} \chi_\mathrm{out}(x)^2 \,\D x \nonumber \\
& \le |\xi_+| \,\phi_+^2\, \|\chi_\mathrm{in}\|^2_{L^2((-\infty,-\hat{s}]\cup [\hat{s},\infty))} \label{Kestplus}
\end{align}
 %----------------%
and
 %----------------%
\begin{align}
\int_{\Omega_-^\mathrm{out} \cap\{K_+\cup K_-\}} & \big\{ |\nabla\phi(x)|^2 -\mu|\phi(x)|^2\big\} \chi_\mathrm{out}(x)^2 \,\D x \nonumber \\
& \le \xi_- \|\chi_\mathrm{in}\|^2_{L^2((-\infty,-\hat{s}]\cup [\hat{s},\infty))}.  \label{Kestminus}
\end{align}
\end{subequations}
 %----------------%
So far we have not employed the convexity of $\Omega_+$; we will need it from now on to estimate the integrals \eqref{Kest}. As before we will first prove the second claim of Theorem~\ref{th:main1} under the additional assumption \ref{s4} using again the notation introduced in Fig.~4.

The part $\Omega_\mathrm{out}^{(-)}$ consists then of a finite number of sectors $\omega_{2j}$ which in view of the convexity assumption do not overlap mutually.
\col{Moreover, since $\Gamma$ is a $C^1$ curve, the neighbouring sectors have common boundaries which are the halflines normal to $\Gamma$ at the points where the curvature changes value, and as a result, the closures of sectors $\omega_{2j}$ cover the region $\Omega_\mathrm{out}^{(-)}$}. Let $\Gamma_\pm$ and $\omega_{k\pm},\,k=1,2$, be the same as in part (a) of Theorem~\ref{th:main2}. By the same reasoning as in the proof of the latter, cf.~\eqref{est12b}, one can check that the contribution of the regions $\omega_{k\pm} \setminus\{K_+\cup K_-\}$ to the integrals \eqref{Kest} can be made arbitrarily small by choosing $r_0$ sufficiently large. Using further the fact that $|\chi_\mathrm{out}|\le 1$ in combination with \eqref{int2}, we get
 %----------------%
\begin{align}
 & \int_{\Omega_\mathrm{out}^{(-)}\setminus\{K_+\cup K_-\}} \big\{ |\nabla\phi(x)|^2 -\mu|\phi(x)|^2\big\} \chi_\mathrm{out}(x)^2 \,\D x \label{tildeminus} \\
 & \qquad \le 2|\mu| \Big[ \frac{2\hat{s}}{\:2\xi_-} + \frac{a}{2\xi_-} \int_{-\hat{s}}^{\hat{s}} \kappa(s)\,\D s + \frac{1}{4\xi_-^2} \int_{-\hat{s}}^{\hat{s}} \kappa(s)\,\D s \Big] + \mathcal{O}(\ee^{-cr_0}) \nonumber \\
 & \qquad = 2\xi_- \hat{s} + a\xi_- \int_{-\hat{s}}^{\hat{s}} \kappa(s)\,\D s + \frac{1}{2} \int_{-\hat{s}}^{\hat{s}} \kappa(s)\,\D s + \mathcal{O}(\ee^{-cr_0}) \nonumber
\end{align}
 %----------------%
for some $c>0$, and since $\kappa(s)=0$ for $|s|>\hat{s}$ we can let the variable $s$ in the above integrals run over the whole $\R$. Comparing now the right-hand side of \eqref{tildeminus} with that of \eqref{inner}, we see that the terms containing $\xi_-$ in the latter have their counterparts here with the opposite sign, hence they cancel mutually.

Next we estimate the contribution to \eqref{Kestplus} coming from $\Omega_\mathrm{out}^{(+)}$. We note that $|\nabla\phi|^2=(-\mu+V_0)|\phi|^2=|\xi_+|^2|\phi(x)|^2$ holds almost everywhere in $\Omega_+^\mathrm{out}$ which means that the integral at the right-hand side of \eqref{outform+} can be rewritten as $2|\xi_+|^2 \int_{\Omega_+^\mathrm{out}} |\phi(x)|^2 \chi_\mathrm{out}(x)^2 \,\D x$. In analogy with \eqref{est10} we can estimate the function $\phi$ using local extrema of the distance function, namely
 %----------------%
\begin{align}
& |\phi(x)|^2 = \phi_+^2\,\exp\{-2|\xi_+|(d_x(s_x^0)-a)\} \label{est15} \\ & \le \phi_+^2 \Big[-\!\!\sum_{s_x^i\in M_x^\uparrow}\! \exp\{-2|\xi_+|(d_x(s_x^i)\!-\!a)\} +\!\! \sum_{s_x^i\in M_x^\downarrow}\! \exp\{-2|\xi_+|(d_x(s_x^i)\!-\!a)\} \Big]. \nonumber
\end{align}
 %----------------%
As in part (a) of Theorem~\ref{th:main2}, we want to replace the integral of the expression at the right-hand side of \eqref{est15} over $\Omega_\mathrm{out}^{(+)}\setminus\{K_+\cup K_-\}$ by the sum of the integrals over the regions $\omega_{1j}$ and $\omega_{3j}$ corresponding to the partition of the curve segment with $s\in[-\hat{s},\hat{s}]$ into circular arcs. In analogy with relation \eqref{est12} we get
 %----------------%
\begin{align}
& \int_{\Omega_\mathrm{out}^{(+)}\setminus\{K_+\cup K_-\}} |\phi(x)|^2\, \D x \label{est13}\\
& \le \phi_+^2 \sum_j \Big\{ \int_{\omega_{1j}\cap\{\Omega_\mathrm{out}^{(+)}\setminus\{K_+\cup K_-\}\}} \exp\{-2|\xi_+|(\mathrm{dist}(x,\Gamma_j)-a)\} \,\D x \nonumber \\
& \quad - \int_{\omega_{3j}\cap\{\Omega_\mathrm{out}^{(+)}\setminus\{K_+\cup K_-\}\}} \exp\{-2|\xi_+|(|\kappa_j|^{-1}+\mathrm{dist}(x,O_j)-a)\} \,\D x \Big\}, \nonumber
\end{align}
 %----------------%
where in contrast to \eqref{est12} the right-hand side \eqref{est13} does not involve integrals over $\omega_{2j}$ because in view of the convexity assumption we have $\Omega_\mathrm{out}^{(+)}\cap \omega_{2j} =\emptyset$ holds for any $j$.

Following the strategy used in the proof of part (a) of Theorem~\ref{th:main2}, we want to replace integrals over $\omega_{kj}\cap\{\Omega_\mathrm{out}^{(+)}\setminus\{K_+\cup K_-\}\},\, k=1,3$, with those over the extended regions $\omega_{kj}\setminus\{K_+\cup K_-\}$, respectively. To this aim, we employ the following simple geometric result:
 %----------------%
\begin{lemma} \label{l:justomega1}
Suppose that $x\in\Omega_-$ does not belong to the boundaries of $\omega_{kj},\,k=1,2,3$, for any $j$. Let further the distance function $d_x(s)$ reach a minimum which is not global at a point of the curve belonging to an arc $\Gamma_{j^*}$, then we have $x\in\omega_{1j^*}$.
\end{lemma}
 %----------------%
\noindent The lemma in fact says that if $\Omega_+$ is convex, it cannot happen that $x\in\omega_{2j^*}$, which is obviously equivalent to the following claim: \\[.5em]
 %----------------%
% \begin{lemma} %\label{l:curvebank}
\noindent \textbf{Lemma 5.1\hspace{-1pt}'.}
\emph{Let $x\in\Omega_-$. For any distance function extremum, except the global minimum, the segment $L_x^i$ connecting the points $x$ and $\Gamma(s_x^i)$ approaches the curve from the side of $\Omega_+$.}
%\end{lemma}
 %----------------%
\begin{proof}
The point of global minimum is obviously approached for the region where $x$ lies, that is, from $\Omega_-$. The next two extrema on both sides of $s_x^0$, provided they exist, are necessarily maxima, and in view of the assumed convexity of $\Omega_+$ the segments $L_x^i$ cannot approach $\Gamma(s_x^i)$ from the side of $\Omega_-$. We denote by $L(s)$ the segment connecting the point $x$ with $\Gamma(s)$. The side from which $L(s)$ approaches the curve can change only at the points where the angle $\beta(s)$ between the segment $L(s)$ and $L_x^0$ corresponding to the global minimum of $d_x(\cdot)$ has, as a function of~$s$, a local maximum or minimum. Since the curve $\Gamma$ is by assumption \mbox{$C^1$-smooth}, and so is $\beta(\cdot)$, the lines connecting such points with $x$ are tangent to $\Omega$, however, a convex region cannot cross its own tangent, hence the extrema of the function $\beta(\cdot)$ are global, one maximum and one minimum. The corresponding points $s_x^i$, provided both of them exist, lie on both sides of $s_x^0$ because a convex region can have only two tangents passing through an exterior point $x$ and the point $\Gamma(s_x^i)$ lies between the two tangent points on the boundary of $\Omega_+$. The same tangent argument shows that once the $L(s)$ switches the side from which it approached $\Gamma$ it can never come back.
\end{proof}
 %----------------%

As before all the local extrema of $d_x(\cdot)$ for $x\in\Omega_-$ except the global minimum come in pairs, so in analogy with \eqref{est12} we are able to estimate the expression \\ $\phi_+^2 \int_{\Omega_-\setminus\{K_+\cup K_-\}} \exp\{-2|\xi_+|(d_x(s_x^0)-a)\}\,\D x$ from above by
 %----------------%
\begin{align}
& \phi_+^2 \sum_j \Big\{ \int_{\omega_{1j}\cap \{\Omega_- \setminus\{K_+\cup K_-\}\}} \exp\{-2|\xi_+|(\mathrm{dist}(x,\Gamma_j)-a)\} \,\D x \label{est14} \\
& - \int_{\omega_{3j}\cap \{\Omega_- \setminus\{K_+\cup K_-\}\}} \exp\{-2|\xi_+|(|\kappa_j|^{-1}\!+\mathrm{dist}(x,O_j)-a)\} \,\D x \Big\}, \nonumber
\end{align}
 %----------------%
where in view of Lemma~\ref{l:justomega1} the first part does not include integration over $\omega_{2j}\cap \{\Omega_- \setminus\{K_+\cup K_-\}\}$. Adding \eqref{est14} to \eqref{est13}, we get
 %----------------%
\begin{align}
& \int_{\Omega_\mathrm{out}^{(+)}\setminus\{K_+\cup K_-\}} |\phi(x)|^2\, \D x \label{est16}\\
& \le \phi_+^2 \sum_j \Big\{ \int_{\omega_{1j}\cap\tilde\Omega} \exp\{-2|\xi_+|(\mathrm{dist}(x,\Gamma_j)-a)\} \,\D x \nonumber \\
& \quad - \int_{\omega_{3j}\cap\tilde\Omega} \exp\{-2|\xi_+|(|\kappa_j|^{-1}+\mathrm{dist}(x,O_j)-a)\} \,\D x \Big\}, \nonumber
\end{align}
 %----------------%
where $\tilde\Omega:= \Omega_- \cup \{\Omega_\mathrm{out}^{(+)}\setminus\{K_+\cup K_-\}\}$. Moreover, applying again the argument that lead to \eqref{estK3} we infer that one can replace $\omega_{1j}\cap\tilde\Omega$ and $\omega_{3j}\cap\tilde\Omega$ in \eqref{est16} by $\omega_{1j}\setminus\{K_+\cup K_-\}$ and $\omega_{3j}\setminus\{K_+\cup K_-\}$, respectively, with an error which can be made arbitrarily small by choosing $r_0$ large enough. The rest of the proof of part (b) of Theorem~\ref{th:main2} for a convex $\Omega_+$ repeats the corresponding part of the proof of the part (a); in the final step we take into account that a convex $\Gamma$ can be approximated by convex curves of piecewise constant curvature \col{so we van proceed as in Sec.~4.3; note that by \eqref{ss_approx} the curvature of $\Gamma$ is approximated pointwise by the those of curves $\hat\Gamma_n$}.

To complete the proof of part (b) of Theorem~\ref{th:main2}, assume next that $\Omega_+$ is \emph{concave}. This case is already easy given the fact that in the first part of the proof we have not used the difference between $|\xi_+|$ and $|\xi_-|$, or between $\phi_+$ and $\phi_-$; the latter was set to one for convenience only. The role of the convexity was just to help us to distinguish the extrema of the distance function referring to the two outer parts of the trial function; if $\Omega_-$ is convex, we can repeat the argument step by step interchanging the roles of $\Omega_-$ and $\Omega_+$ arriving thus at the sought claim.

It remains to prove part (b) of Theorem~\ref{th:main1} where we have $\mu=0$ by assumption and $\Omega_+$ is again convex. Since $V_0>0$, the equation $h\phi=0$ has a resonance solution $\phi_0$ which is constant for $t\le-a$ and decays exponentially for $t>a$; as before we normalize it putting $\phi_-=1$. We have to construct a trial function $\psi\in H^2(\R^2)$ which makes the quadratic form \eqref{form}, now containing the potential bias, negative. We use elements of the previous proofs. In particular, inside $\Omega^a$ the function will be given by \eqref{trial_in} and \eqref{mollif}. Outside $\Omega^a$ the trial function in $\Omega_-$ will be the same as in the proof of part (a) of Theorem~\ref{th:main1}, \col{cf.~Sec.3.2}, while in $\Omega_+$ we choose it as in the of part (b) of Theorem~\ref{th:main2} discussed above, putting there $\mu=0$, in other words, as \eqref{outer2} in which in view of \eqref{shorthands} we set $\xi_+=-\sqrt{V_0}$. Repeating then the estimates used to prove part (a) Theorem~\ref{th:main1} in $\Omega_-$ and part (a) of Theorem~\ref{th:main2} in $\Omega_+$, we obtain
 %----------------%
\begin{equation}\label{biasedform2}
Q[\psi] = -\frac18\delta\nu - \int_\R \kappa(s)\,\D s + o(\psi),
\end{equation}
 %----------------%
where the error term can be made arbitrarily small by choosing large $r_0$ and $s^*$ in \eqref{mollif}. In view of the assumed convexity of $\Omega_+$ we have $\int_\R \kappa(s)\,\D s>0$, hence choosing the parameters properly we can make the form negative; this concludes the proof of part (b) of Theorem~\ref{th:main1}.

 %----------------%
\begin{remark}
As we have noted in the introduction, the `two-sided' validity of part (b) Theorem~\ref{th:main2} does not extend to the zero-energy resonance case. The above proof indicates the source of this difference. While for $\mu<0$ we can use the trial function from the proof of part (b) of Theorem~\ref{th:main1} and simply switch the roles of $\Omega_+$ and $\Omega_-$, a similar interchange does not work if $\mu=0$ because it leads to the sign change of the second term on the right-hand side of \eqref{biasedform2} and we are obviously not free to choose $\delta\nu$ to compensate this positive number.
\end{remark}
 %----------------%

%------
% Insert acknowledgments and information
% regarding funding at the end of the last
% section, i.e., right before the bibliography.
%------

\begin{ack}
We thank the referee for useful comments.
\end{ack}

\begin{funding}
The work of P.E. was supported by the Czech Science Foundation within the project 21-07129S. S.V. was funded by Deutsche Forschungsgemeinschaft-Project-ID 258734477-SFB-1173.
\end{funding}

%------
% Insert the bibliography.
%------

%%%%%%%%%%%%%%%%%%%%%%%%%%%%%%%%%%%%%%%%%%%%%%%%%%%%%%%%%%%%%%%%%%

\end{document}